\documentclass[11pt]{article}
\usepackage[a4paper,top=3cm,bottom=3.5cm,left=2.5cm,right=2.5cm,marginparwidth=1.75cm]{geometry}
\usepackage[english]{babel}
\usepackage[utf8x]{inputenc}
\usepackage[T1]{fontenc}
\usepackage{graphicx}
\usepackage{caption, subcaption}
\usepackage[title]{appendix}
\usepackage{tikz}
\usepackage{pgfplots}

\usepackage{cutwin}
\graphicspath{ {images/} }
\usetikzlibrary{intersections,decorations.pathreplacing,decorations.markings,calc,angles,quotes,arrows.meta,pgfplots.fillbetween,patterns}

\usepackage{amsmath,yhmath}
\usepackage{amsthm}
\usepackage{amssymb}
\usepackage{hyperref}
\usepackage{cleveref}
\usepackage{epstopdf}
\usepackage{comment}
\usepackage{todonotes}
\usepackage{color}
\usepackage{float}
\usepackage[sort&compress,numbers]{natbib}
\newtheorem{thm}{Theorem}
\newtheorem{lem}[thm]{Lemma}
\newtheorem{prop}[thm]{Proposition}
\newtheorem{corollary}[thm]{Corollary}
\theoremstyle{definition}
\newtheorem{defn}[thm]{Definition}
\newtheorem{example}[thm]{Example}

\newtheorem{rmk}[thm]{Remark}
\numberwithin{equation}{section}
\numberwithin{thm}{section}

\newcommand{\N}{\mathbb{N}}

\newcommand{\R}{\mathbb{R}}

\renewcommand{\S}{\mathcal{S}}

\newcommand{\Z}{\mathbb{Z}}
\renewcommand{\r}{\mathrm{r}}
\newcommand{\p}{\partial}
\newcommand{\dx}{\mathrm{d}}

\newcommand{\nm}{\noalign{\smallskip}}
\newcommand{\ds}{\displaystyle}
\newcommand{\iu}{\mathrm{i}\mkern1mu}

\newcommand{\sddots}{\raisebox{3pt}{\scalebox{.6}{$\ddots$}}}

\makeatletter
\newcommand{\neutralize}[1]{\expandafter\let\csname c@#1\endcsname\count@}
\makeatother
\title{Non-reciprocal wave propagation in space-time modulated media\thanks{\footnotesize
This work was supported in part by the Swiss National Science Foundation grant number
200021--200307.}}

\author{
	Habib Ammari\thanks{\footnotesize Department of Mathematics, 
		ETH Z\"urich, 
		R\"amistrasse 101, CH-8092 Zurich, Switzerland (habib.ammari@math.ethz.ch, jinghao.cao@sam.math.ethz.ch).} \and Jinghao Cao\footnotemark[2] \and Erik Orvehed Hiltunen\thanks{\footnotesize Department of Mathematics, Yale University, 51 Prospect Street, New Haven CT 06511, USA  (erik.hiltunen@yale.edu).}}

\date{}

\begin{document}
	\maketitle
	
	\begin{abstract}
		We prove the possibility of achieving non-reciprocal wave propagation in space-time modulated media and give an asymptotic analysis of the non-reciprocity property in terms of the amplitude of the time-modulation. Such modulation causes a folding of the band structure of the material, which may induce degenerate points. By breaking time-reversal symmetry, we show that these degeneracies may open into non-symmetric, unidirectional band gaps. Finally, we illustrate our results by several numerical simulations.
	\end{abstract}

\noindent{\textbf{Mathematics Subject Classification (MSC2000):} 35J05, 35C20, 35P20, 74J20
		
\vspace{0.2cm}
		
\noindent{\textbf{Keywords:}} non-reciprocal wave propagation, unidirectional wave, subwavelength quasifrequency, time-modulation, space-time modulated medium, metamaterial, band gap structure
	\vspace{0.5cm}


\section{Introduction}
The control and manipulation of wave-matter interactions at subwavelength scales
has received considerable attention over the
past decade \cite{lemoult2016soda,yves2017crytalline,
phononic1,phononic2}. Moreover, the potential for using artificially
structured metamaterials has shown considerable 
promise \cite{ammari2021functional,review,review2}. 
Here, \emph{subwavelength} means that the length-scale of the system is considerably smaller than the operating wavelength. Subwavelength metamaterials can be achieved by having a locally resonant microstructure. In other words, the material is composed of building-blocks which themselves are subwavelength resonators \cite{ammari2017subwavelength,yves2017crystalline,yves2017topological,wang2019subwavelength}.

As reviewed in \cite{ammari2021functional}, \emph{high-contrast} resonators are a natural choice of resonators when designing subwavelength metamaterials. Here, the subwavelength nature stems from a high material contrast between the constituting materials of the structure. Such structures can be used to achieve a variety of effects \cite{davies2019fully, ammari2018minnaert,ammari2020exceptional,ammari2020highfrequency,ammari2020highorder,ammari2017subwavelength,ammari2017double,ammari2020honeycomb,MaCMiPaP}. 
Of particular importance for us are systems of time dependent high-contrast resonators. In \cite{ammari2020time}, such systems are studied and a mathematical foundation that explains some effects found in time-modulated systems for waves in the subwavelength frequency regime is provided.

In the past, significant progress has been achieved
in the field of classical waves by making use of analogies with
electronic systems \cite{analog1,analog2}.  For instance, the idea of a band gap material,
a system with a spatially varying and periodic material parameters, was motivated by the well-known physics of electronic Bloch states; the  scattering of waves in periodic media presents the same formal solutions as those for
the scattering of electrons in periodic potentials. 
 More recently, the field of topological insulators in condensed matter physics has been teeming with intriguing and very exciting discoveries. Notably, the capacity of guiding currents towards specific directions according to the spin of the travelling electrons has a great potential for electronic devices \cite{hallbook,alexis}. 

Several attempts to transpose this phenomenon to classical waves at subwavelength regimes, unveiling the pseudo-spin locking of guided waves have been made; see, for instance,  \cite{fleury2016floquet,rechtsman2013photonic,raghu2008analogs,
nash2015topological,nassar2018quantization,wilson2018temporal,wilson2019temporally}. In order to replicate spin effects from quantum systems, time-reversal symmetry should be broken. 
However, classical (nondissipative) systems are invariant under
time reversal because their dynamics are governed by the wave equation, which, unlike the Schrödinger equation, is second
order in time.

Reciprocity is an expression of time-reversal symmetry. It is a fundamental principle in wave physics, requiring that the response of a transmission channel is symmetric when source and observation points are interchanged. It is of major significance because it poses fundamental constraints on the way we
process acoustic, elastic and electromagnetic signals \cite{alu1}.  Recent trends for subwavelength devices and technological advances in the realization of efficient time-modulated systems have recently brought
time-modulated non-reciprocal devices to the spotlight. Over the
past decade, non-reciprocity based on time modulation has gained
significant attention  for different physical systems, such as in acoustics, 
mechanics, and optics \cite{alu1}. 

In this paper, we discuss the most fundamental mechanisms 
of non-reciprocity in metamaterials based on time modulation. By using lattices of spatiotemporally modulated subwavelength resonators where time-reversal symmetry is broken, we prove
the unidirectional excitation of waves guided at subwavelength scales. In the presence of only spatial modulation, the time-reversal symmetry is not broken and consequently the band functions are  symmetric for opposite directions. Breaking time-reversal symmetry, we show that
time-modulation may open degenerate points of the folded band structure into non-symmetric band gaps; opposite propagation directions are subject to distinct band gaps. If the excitation frequency falls inside the band gap for only one propagation direction, wave transmission is then prohibited in
this direction but not in the opposite one, leading to non-reciprocal transmission properties. 


Our results in this paper use the fundamental fact that phase-shifted (``rotation like'') time-modulations of subwavelength resonators can provide a kind of ``artificial spin''. They show that unidirectional guiding phenomenon is not particular to quantum systems, as conjectured in the seminal papers \cite{haldane2,haldane}. Such artificial spin cannot be achieved in systems of one or two resonators. In fact, we show that non-reciprocity requires at least three resonators inside one unit cell of the material.

%
%

This paper is organized as follows. In Section 2, we define the problem of reciprocity and discuss the Floquet-Bloch theory which is essential to solve ordinary differential equations with periodic coefficients. In Section 3, we discuss conditions in the 
time-modulation which preserve the property of reciprocity. Section 4 is devoted to the asymptotic analysis of the non-reciprocity in terms of the amplitude of the modulation. In Section 5, we numerically simulate non-reciprocity properties in a variety of structures. The paper ends with some concluding remarks in Section \ref{conclusion}.

\section{Problem formulation and preliminary theory}	
In this section, we define the problem of reciprocity. Moreover, we introduce the Floquet-Bloch theory to apply to the problem. This subsection follows closely the introductory theory provided in \cite{ammari2020time}.

\subsection{Resonator structure and wave equation}
\label{sec:formulation}
We consider the wave equation in structures with time-modulated materials. Such wave equation can be used to model acoustic and polarized electromagnetic waves.
The time dependent material parameters are given by $\rho(x,t)$ and $\kappa(x,t)$. 
In acoustics, $\rho$ and $\kappa$ represent the density and the bulk modulus of the materials. We study the time-dependent wave equation in dimension $d=2,3$:
\begin{equation}
\label{waveequation}
\left(\frac{\partial}{\partial t }\frac{1}{\kappa(x,t)}\frac{\partial}{\partial t}-\nabla\cdot\frac{1}{\rho(x,t)}\nabla\right)u(x,t)=0,\ \ x\in \mathbb{R}^d,t\in\mathbb{R}.	
\end{equation}
Furthermore, we assume a fully periodic resonator structure with a lattice $\Lambda \subset \R^d$ and unit cell $Y\subset \R^d$. Each unit cell contains a system of $N$ resonators $D\Subset Y$. $D$ is constituted by $N$ disjoint domains $D_i$ for $i=1,\ldots,N$, each $D_i$ being connected and having boundary of Hölder class $\p D_i \in C^{1,s}, 0 < s < 1$.  Denote $\mathcal{C}_i$ and $\mathcal{C}$ the periodically repeated $i^{\text{th}}$ resonators and the full crystal:
\begin{equation*}
	\mathcal{C}_i=\bigcup_{m\in \Lambda}D_i+m,\ \ \mathcal{C}=\bigcup_{m\in\Lambda}D+m.
\end{equation*} 	
We let $\Lambda^*$ to be the dual lattice and define the (space-) Brillouin zone $Y^*$ as the torus $Y^*:= \mathbb{R}^d/\Lambda^*$. 

\begin{figure}[H]
		\begin{subfigure}[b]{0.45\linewidth}
			\centering
			
			\begin{tikzpicture}[scale=1.5]
				\begin{scope}[scale=1]
					\pgfmathsetmacro{\rb}{0.13pt}
					\pgfmathsetmacro{\rs}{0.1pt}
					\pgfmathsetmacro{\ao}{326}
					\pgfmathsetmacro{\at}{46}
					\pgfmathsetmacro{\ad}{228}
					\pgfmathsetmacro{\af}{100}
					\coordinate (a) at (1,{1/sqrt(3)});		
					\coordinate (b) at (1,{-1/sqrt(3)});	
					\coordinate (c) at (2,0);
					\draw[->] (0,0) -- (a) node[pos=0.9,xshift=0,yshift=7]{ $l_1$} node[pos=0.5,above]{$Y$};
					\draw[->] (0,0) -- (b) node[pos=0.9,xshift=0,yshift=-5]{ $l_2$};
					\draw[opacity=0.5] (a) -- (c) -- (b);
					\begin{scope}[xshift = 1.1cm, yshift=0.28cm,rotate=\ao]
						\draw[fill=lightgray] plot [smooth cycle] coordinates {(0:\rb) (60:\rs) (120:\rb) (180:\rs) (240:\rb) (300:\rs) };
					\end{scope}
					\begin{scope}[xshift = 0.5cm, rotate=\at]
						\draw[fill=lightgray] plot [smooth cycle] coordinates {(0:\rb) (60:\rs) (120:\rb) (180:\rs) (240:\rb) (300:\rs) };
					\end{scope}
					\begin{scope}[xshift = 1.4cm, yshift=-0.14cm, rotate=\ad]
						\draw[fill=lightgray] plot [smooth cycle] coordinates {(0:\rb) (60:\rs) (120:\rb) (180:\rs) (240:\rb) (300:\rs) };
					\end{scope}
					\begin{scope}[xshift = 0.9cm,yshift=-0.16cm, rotate=\af]
						\draw[fill=lightgray] plot [smooth cycle] coordinates {(0:\rb) (60:\rs) (120:\rb) (180:\rs) (240:\rb) (300:\rs) };
					\end{scope}
				\end{scope}
			\end{tikzpicture}
			\vspace{0.65cm}
			\caption{Unit cell $Y$ containing $N=4$ resonators.}
		\end{subfigure}
		\begin{subfigure}[b]{0.5\linewidth}
			
			\begin{tikzpicture}[scale=1]
			
				\begin{scope}[xshift=-5cm,scale=1]

					\coordinate (a) at (1,{1/sqrt(3)});		
					\coordinate (b) at (1,{-1/sqrt(3)});	
					\coordinate (c) at (2,0);
					\pgfmathsetmacro{\rb}{0.13pt}
					\pgfmathsetmacro{\rs}{0.1pt}
					\pgfmathsetmacro{\ao}{326}
					\pgfmathsetmacro{\at}{46}
					\pgfmathsetmacro{\ad}{228}
					\pgfmathsetmacro{\af}{100}	
					\draw[opacity=0.2] (0,0) -- (a);
					\draw[opacity=0.2] (0,0) -- (b);
					\draw[opacity=0.2] (a) -- (c) -- (b);
					\begin{scope}[xshift = 1.1cm, yshift=0.28cm,rotate=\ao]
						\draw[fill=lightgray] plot [smooth cycle] coordinates {(0:\rb) (60:\rs) (120:\rb) (180:\rs) (240:\rb) (300:\rs) };
				\end{scope}
					
				\begin{scope}[xshift = 0.5cm, rotate=\at]
				
						\draw[fill=lightgray] plot [smooth cycle] coordinates {(0:\rb) (60:\rs) (120:\rb) (180:\rs) (240:\rb) (300:\rs) };
					\end{scope}
					\begin{scope}[xshift = 1.4cm, yshift=-0.14cm, rotate=\ad]
						\draw[fill=lightgray] plot [smooth cycle] coordinates {(0:\rb) (60:\rs) (120:\rb) (180:\rs) (240:\rb) (300:\rs) };
					\end{scope}
					\begin{scope}[xshift = 0.9cm,yshift=-0.16cm, rotate=\af]
						\draw[fill=lightgray] plot [smooth cycle] coordinates {(0:\rb) (60:\rs) (120:\rb) (180:\rs) (240:\rb) (300:\rs) };
					\end{scope}			
					\begin{scope}[shift = (a)]
						\draw[opacity = 0.2] (0,0) -- (1,{1/sqrt(3)}) -- (2,0) -- (1,{-1/sqrt(3)}) -- cycle; 
						\begin{scope}[xshift = 1.1cm, yshift=0.28cm,rotate=\ao]
							\draw[fill=lightgray] plot [smooth cycle] coordinates {(0:\rb) (60:\rs) (120:\rb) (180:\rs) (240:\rb) (300:\rs) };
						\end{scope}
						\begin{scope}[xshift = 0.5cm, rotate=\at]
							\draw[fill=lightgray] plot [smooth cycle] coordinates {(0:\rb) (60:\rs) (120:\rb) (180:\rs) (240:\rb) (300:\rs) };
						\end{scope}
						\begin{scope}[xshift = 1.4cm, yshift=-0.14cm, rotate=\ad]
							\draw[fill=lightgray] plot [smooth cycle] coordinates {(0:\rb) (60:\rs) (120:\rb) (180:\rs) (240:\rb) (300:\rs) };
						\end{scope}
						\begin{scope}[xshift = 0.9cm,yshift=-0.16cm, rotate=\af]
							\draw[fill=lightgray] plot [smooth cycle] coordinates {(0:\rb) (60:\rs) (120:\rb) (180:\rs) (240:\rb) (300:\rs) };
						\end{scope}
					\end{scope}
					\begin{scope}[shift = (b)]
						\draw[opacity = 0.2] (0,0) -- (1,{1/sqrt(3)}) -- (2,0) -- (1,{-1/sqrt(3)}) -- cycle; 
						\begin{scope}[xshift = 1.1cm, yshift=0.28cm,rotate=\ao]
							\draw[fill=lightgray] plot [smooth cycle] coordinates {(0:\rb) (60:\rs) (120:\rb) (180:\rs) (240:\rb) (300:\rs) };
						\end{scope}
						\begin{scope}[xshift = 0.5cm, rotate=\at]
							\draw[fill=lightgray] plot [smooth cycle] coordinates {(0:\rb) (60:\rs) (120:\rb) (180:\rs) (240:\rb) (300:\rs) };
						\end{scope}
						\begin{scope}[xshift = 1.4cm, yshift=-0.14cm, rotate=\ad]
							\draw[fill=lightgray] plot [smooth cycle] coordinates {(0:\rb) (60:\rs) (120:\rb) (180:\rs) (240:\rb) (300:\rs) };
						\end{scope}
						\begin{scope}[xshift = 0.9cm,yshift=-0.16cm, rotate=\af]
							\draw[fill=lightgray] plot [smooth cycle] coordinates {(0:\rb) (60:\rs) (120:\rb) (180:\rs) (240:\rb) (300:\rs) };
						\end{scope}
					\end{scope}
					\begin{scope}[shift = ($-1*(a)$)]
						\draw[opacity = 0.2] (0,0) -- (1,{1/sqrt(3)}) -- (2,0) -- (1,{-1/sqrt(3)}) -- cycle; 
						\begin{scope}[xshift = 1.1cm, yshift=0.28cm,rotate=\ao]
							\draw[fill=lightgray] plot [smooth cycle] coordinates {(0:\rb) (60:\rs) (120:\rb) (180:\rs) (240:\rb) (300:\rs) };
						\end{scope}
						\begin{scope}[xshift = 0.5cm, rotate=\at]
							\draw[fill=lightgray] plot [smooth cycle] coordinates {(0:\rb) (60:\rs) (120:\rb) (180:\rs) (240:\rb) (300:\rs) };
						\end{scope}
						\begin{scope}[xshift = 1.4cm, yshift=-0.14cm, rotate=\ad]
							\draw[fill=lightgray] plot [smooth cycle] coordinates {(0:\rb) (60:\rs) (120:\rb) (180:\rs) (240:\rb) (300:\rs) };
						\end{scope}
						\begin{scope}[xshift = 0.9cm,yshift=-0.16cm, rotate=\af]
							\draw[fill=lightgray] plot [smooth cycle] coordinates {(0:\rb) (60:\rs) (120:\rb) (180:\rs) (240:\rb) (300:\rs) };
						\end{scope}
					\end{scope}
					\begin{scope}[shift = ($-1*(b)$)]
						\draw[opacity = 0.2] (0,0) -- (1,{1/sqrt(3)}) -- (2,0) -- (1,{-1/sqrt(3)}) -- cycle; 
						\begin{scope}[xshift = 1.1cm, yshift=0.28cm,rotate=\ao]
							\draw[fill=lightgray] plot [smooth cycle] coordinates {(0:\rb) (60:\rs) (120:\rb) (180:\rs) (240:\rb) (300:\rs) };
						\end{scope}
						\begin{scope}[xshift = 0.5cm, rotate=\at]
							\draw[fill=lightgray] plot [smooth cycle] coordinates {(0:\rb) (60:\rs) (120:\rb) (180:\rs) (240:\rb) (300:\rs) };
						\end{scope}
						\begin{scope}[xshift = 1.4cm, yshift=-0.14cm, rotate=\ad]
							\draw[fill=lightgray] plot [smooth cycle] coordinates {(0:\rb) (60:\rs) (120:\rb) (180:\rs) (240:\rb) (300:\rs) };
						\end{scope}
						\begin{scope}[xshift = 0.9cm,yshift=-0.16cm, rotate=\af]
							\draw[fill=lightgray] plot [smooth cycle] coordinates {(0:\rb) (60:\rs) (120:\rb) (180:\rs) (240:\rb) (300:\rs) };
						\end{scope}
					\end{scope}
					\begin{scope}[shift = ($(a)+(b)$)]
						\draw[opacity = 0.2] (0,0) -- (1,{1/sqrt(3)}) -- (2,0) -- (1,{-1/sqrt(3)}) -- cycle; 
						\begin{scope}[xshift = 1.1cm, yshift=0.28cm,rotate=\ao]
							\draw[fill=lightgray] plot [smooth cycle] coordinates {(0:\rb) (60:\rs) (120:\rb) (180:\rs) (240:\rb) (300:\rs) };
						\end{scope}
						\begin{scope}[xshift = 0.5cm, rotate=\at]
							\draw[fill=lightgray] plot [smooth cycle] coordinates {(0:\rb) (60:\rs) (120:\rb) (180:\rs) (240:\rb) (300:\rs) };
						\end{scope}
						\begin{scope}[xshift = 1.4cm, yshift=-0.14cm, rotate=\ad]
							\draw[fill=lightgray] plot [smooth cycle] coordinates {(0:\rb) (60:\rs) (120:\rb) (180:\rs) (240:\rb) (300:\rs) };
						\end{scope}
						\begin{scope}[xshift = 0.9cm,yshift=-0.16cm, rotate=\af]
							\draw[fill=lightgray] plot [smooth cycle] coordinates {(0:\rb) (60:\rs) (120:\rb) (180:\rs) (240:\rb) (300:\rs) };
						\end{scope}
					\end{scope}
					\begin{scope}[shift = ($-1*(a)-(b)$)]
						\draw[opacity = 0.2] (0,0) -- (1,{1/sqrt(3)}) -- (2,0) -- (1,{-1/sqrt(3)}) -- cycle; 
						\begin{scope}[xshift = 1.1cm, yshift=0.28cm,rotate=\ao]
							\draw[fill=lightgray] plot [smooth cycle] coordinates {(0:\rb) (60:\rs) (120:\rb) (180:\rs) (240:\rb) (300:\rs) };
						\end{scope}
						\begin{scope}[xshift = 0.5cm, rotate=\at]
							\draw[fill=lightgray] plot [smooth cycle] coordinates {(0:\rb) (60:\rs) (120:\rb) (180:\rs) (240:\rb) (300:\rs) };
						\end{scope}
						\begin{scope}[xshift = 1.4cm, yshift=-0.14cm, rotate=\ad]
							\draw[fill=lightgray] plot [smooth cycle] coordinates {(0:\rb) (60:\rs) (120:\rb) (180:\rs) (240:\rb) (300:\rs) };
						\end{scope}
						\begin{scope}[xshift = 0.9cm,yshift=-0.16cm, rotate=\af]
							\draw[fill=lightgray] plot [smooth cycle] coordinates {(0:\rb) (60:\rs) (120:\rb) (180:\rs) (240:\rb) (300:\rs) };
						\end{scope}
					\end{scope}
					\begin{scope}[shift = ($(a)-(b)$)]
						\draw[opacity = 0.2] (0,0) -- (1,{1/sqrt(3)}) -- (2,0) -- (1,{-1/sqrt(3)}) -- cycle; 
						\begin{scope}[xshift = 1.1cm, yshift=0.28cm,rotate=\ao]
							\draw[fill=lightgray] plot [smooth cycle] coordinates {(0:\rb) (60:\rs) (120:\rb) (180:\rs) (240:\rb) (300:\rs) };
						\end{scope}
						\begin{scope}[xshift = 0.5cm, rotate=\at]
							\draw[fill=lightgray] plot [smooth cycle] coordinates {(0:\rb) (60:\rs) (120:\rb) (180:\rs) (240:\rb) (300:\rs) };
						\end{scope}
						\begin{scope}[xshift = 1.4cm, yshift=-0.14cm, rotate=\ad]
							\draw[fill=lightgray] plot [smooth cycle] coordinates {(0:\rb) (60:\rs) (120:\rb) (180:\rs) (240:\rb) (300:\rs) };
						\end{scope}
						\begin{scope}[xshift = 0.9cm,yshift=-0.16cm, rotate=\af]
							\draw[fill=lightgray] plot [smooth cycle] coordinates {(0:\rb) (60:\rs) (120:\rb) (180:\rs) (240:\rb) (300:\rs) };
						\end{scope}
					\end{scope}
					\begin{scope}[shift = ($-1*(a)+(b)$)]
						\draw[opacity = 0.2] (0,0) -- (1,{1/sqrt(3)}) -- (2,0) -- (1,{-1/sqrt(3)}) -- cycle; 
						\begin{scope}[xshift = 1.1cm, yshift=0.28cm,rotate=\ao]
							\draw[fill=lightgray] plot [smooth cycle] coordinates {(0:\rb) (60:\rs) (120:\rb) (180:\rs) (240:\rb) (300:\rs) };
						\end{scope}
						\begin{scope}[xshift = 0.5cm, rotate=\at]
							\draw[fill=lightgray] plot [smooth cycle] coordinates {(0:\rb) (60:\rs) (120:\rb) (180:\rs) (240:\rb) (300:\rs) };
						\end{scope}
						\begin{scope}[xshift = 1.4cm, yshift=-0.14cm, rotate=\ad]
							\draw[fill=lightgray] plot [smooth cycle] coordinates {(0:\rb) (60:\rs) (120:\rb) (180:\rs) (240:\rb) (300:\rs) };
						\end{scope}
						\begin{scope}[xshift = 0.9cm,yshift=-0.16cm, rotate=\af]
							\draw[fill=lightgray] plot [smooth cycle] coordinates {(0:\rb) (60:\rs) (120:\rb) (180:\rs) (240:\rb) (300:\rs) };
						\end{scope}
					\end{scope}
					\begin{scope}[shift = ($2*(a)$)]
						\draw (1,0) node[rotate=30]{$\cdots$};
					\end{scope}
					\begin{scope}[shift = ($-2*(a)$)]
						\draw (1,0) node[rotate=210]{$\cdots$};
					\end{scope}
					\begin{scope}[shift = ($2*(b)$)]
						\draw (1,0) node[rotate=-30]{$\cdots$};
					\end{scope}
					\begin{scope}[shift = ($-2*(b)$)]
						\draw (1,0) node[rotate=150]{$\cdots$};
					\end{scope}
				\end{scope}
			\end{tikzpicture}
			
			\caption{Infinite, periodic system with unit cell $Y$ and 
			lattice $\Lambda$.}
		\end{subfigure}
		\caption{Illustrations of the unit cell and the periodic system of resonators.} 
	\end{figure}
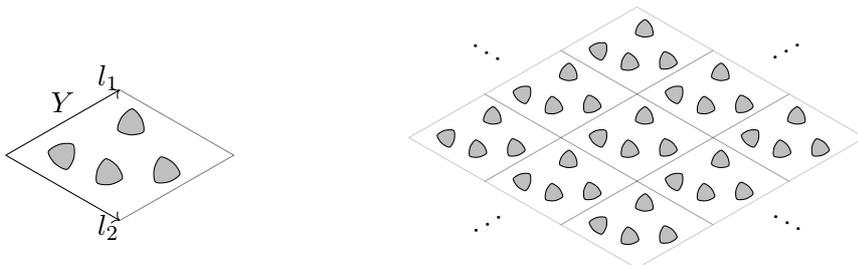
\noindent
For the purpose of this paper, we apply time-modulation to the interior of the resonators, while the surrounding material is constant in $t$. We let
\begin{equation}
	\label{modulation}
	\kappa(x,t)=
\begin{cases}
	\kappa_0, \ & x\in\mathbb{R}^d\backslash \overline{\mathcal{C}}\\
	\kappa_r\kappa_i(t),\ & x\in\mathcal{C}_i
\end{cases} 
,\ \ 
	\rho(x,t)=
\begin{cases}
	\rho_0, \ & x\in\mathbb{R}^d\backslash \overline{\mathcal{C}}\\
	\rho_r\rho_i(t),\ & x\in\mathcal{C}_i
\end{cases},
\end{equation}
for $i=1,\ldots, N$. Here, $\rho_0$, $\kappa_0$, $\rho_r$, and $\kappa_r$ are positive constants. The functions $\rho_i(t) \in \mathcal{C}^0(\mathbb{R})$ and $\kappa_i(t) \in \mathcal{C}^1(\mathbb{R})$ describe the modulation inside the $i^{\text{th}}$ resonator $\mathcal{C}_i$. We assume that each of $\rho_i,\kappa_i$ is periodic with period $T$.

We define the contrast parameter $\delta$ as 
$$
\delta := \frac{\rho_r}{\rho_0}.
$$
In (\ref{waveequation}), we have the transmission conditions at $x\in \p D_i$
$$\delta \frac{\partial {u}}{\partial \nu} \bigg|_{+} - \frac{1}{\rho_i(t)}\frac{\partial {u}}{\partial \nu} \bigg|_{-} = 0, \qquad x\in \p D_i, \ t\in \R,$$
where $\partial/\partial \nu$ is the outward normal derivative at $\p D_i$ and $|_{+,-}$ denote the limits from outside and inside $D_i$, respectively

In order to achieve subwavelength resonances we assume that $\delta \ll 1$ and consider the regime where the modulation frequency $$\Omega := \frac{2\pi}{T} = O(\delta^{1/2}).$$ 
We also assume that $\dx \kappa_i/\dx t= O(\delta^{1/2})$ for $i=1,\ldots,N.$

Note that in the static case where there is no modulation of the material parameters (i.e., when $\rho_i(t)=\kappa_i(t)=1$ for all $i$), the system of $N$ subwavelength resonators has $N$ {\em subwavelength frequencies} of order of $O(\delta^{1/2})$. We refer the reader to \cite{ammari2021functional} for the details.

\subsection{Floquet-Bloch theory}
\label{Floquettheory}
	Let $A(t)$ be a $T$-periodic $N\times N$ complex matrix function and consider the ordinary differential equation (ODE):
	\begin{equation}
	\label{ode}
		\frac{\dx x}{\dx t}(t) =A(t)x(t).
	\end{equation}
Recall that the fundamental solution matrix of (\ref{ode}) is a $N\times N$ matrix with linear independent column vectors, which solves (\ref{ode}). The following theorem is classical. 
\begin{thm}
	(Floquet's theorem) Denote $X(t)$ the matrix-valued fundamental solution with initial value $X(0)=\mathrm{Id}$, where $\mathrm{Id}$ is the identity matrix. There exists a constant matrix $F$ and a $T$-periodic matrix function $P(t)$ such that
	\begin{equation}
	\label{floquetthm}
	X(t)= P(t) e^{Ft}.
	\end{equation}
\end{thm}
\noindent
For each eigenvalue $\lambda := e^{\mathrm{i}\omega}$ of $e^{F}$, there is a Bloch solution $x(t)$ which is $\omega$-quasiperiodic, i.e.,
\begin{equation*}
	x(t+T)=e^{\iu\omega T}x(t).
\end{equation*}
Observe that $\omega$ is defined modulo $\Omega$. Therefore, we define the time-Brillouin zone as $Y_t^*:=\mathbb{C}/(\Omega\mathbb{Z})$.

\begin{rmk}
	In some literatures, e.g. \cite{Yakubovich}, $e^{\mathrm{i}\omega T}$ is called  a characteristic multiplier. We refer to $\omega$ as a quasifrequency, while $\iu \omega$ is a Floquet exponent.
\end{rmk}

If $A$ is time-independent, the solution to \eqref{ode} can be written as $x(t)=e^{At}x(0)$. The Floquet exponents are then given by the eigenvalues of $A$. Since the Floquet exponent is defined modulo $\iu\Omega$, we need the following definition.
\begin{defn}[Folding number]
\label{foldingnumber}
	Let $\omega_A$ be the imaginary part of an eigenvalue of the time independent matrix $A$, we can uniquely write $\omega_A=\omega_0+m\Omega$, where $\omega_0\in [-\Omega/2,\Omega/2)$. The integer $m$ is called the folding number.	
\end{defn}

Applying the Floquet transform to the wave equation (\ref{waveequation}) in  $x$ and seeking quasiperiodic solutions in $t$, we obtain the differential problem
 \begin{equation} \label{eq:wave_transf}
 	\begin{cases}\ \ds \left(\frac{\p }{\p t } \frac{1}{\kappa(x,t)} \frac{\p}{\p t} - \nabla \cdot \frac{1}{\rho(x,t)} \nabla\right) u(x,t) = 0,\\[0.3em]
 		\	u(x,t)e^{-\iu \alpha\cdot x} \text{ is $\Lambda$-periodic in $x$,}\\
 		\	u(x,t)e^{-\iu \omega t} \text{ is $T$-periodic in $t$}. 
 	\end{cases}
 \end{equation} 
For a given $\alpha\in Y^*$, we seek $\omega\in Y_t^*$ such that there is a non-zero solution $u$ to (\ref{eq:wave_transf}).
\begin{defn}
	Quasifrequencies as a function of $\alpha$, i.e., $\alpha \mapsto \omega(\alpha)$ is called a band function. All band functions together constitute  the band structure, or dispersion relationship, of the material.
\end{defn}
The quasiperiodicity (or quasimomentum) $\alpha$ corresponds to the direction of wave propagation, and we therefore introduce the following definition.
\begin{defn}
 	Waves propagate reciprocally if for every $\alpha \in Y^*$, the set of quasifrequencies of  (\ref{eq:wave_transf}) at $\alpha$ coincides with the set of quasifrequencies at $-\alpha$. The reciprocal equation associated with  (\ref{eq:wave_transf}) defined with $\alpha\in Y^*$ is that with $-\alpha$.
\end{defn} 
The purpose of this paper is to investigate under which time-modulation conditions the reciprocity of waves can be broken, and to give an asymptotic analysis of the reciprocity property in terms of the amplitude of the modulation. 
\subsection{Layer-potential theory and the capacitance matrix}
We first define 
	the $\alpha$-quasiperiodic Green's function $G^{\alpha,k}(x,y)$ as the solution of the following equation:
	\begin{equation*}
	\Delta_xG^{\alpha,k}(x,y) + k^2G^{\alpha,k}(x,y) = \sum_{n \in \Lambda} \delta(x-n)e^{\iu\alpha\cdot n}.
	\end{equation*}
It can be shown that
 if $k \neq |\alpha+q|$ for all $q\in \Lambda^*$, then $G^{\alpha,k}$ is given by 
\begin{equation*}
	G^{\alpha,k}(x,y)= \frac{1}{|Y|}\sum_{q\in \Lambda^*} \frac{e^{\iu(\alpha+q)\cdot (x-y)}}{ k^2-|\alpha+q|^2},
\end{equation*}
where $|Y|$ denotes the volume of $Y$;	see, for instance, \cite{MaCMiPaP,ammari2009layer}.

Let $D\subset \R^d$ be as in Section \ref{sec:formulation}. We define the quasiperiodic single layer potential $\mathcal{S}_D^{\alpha,k}: L^2(\partial D) \rightarrow H_{\textrm{loc}}^1(\R^d)$ by
$$\mathcal{S}_D^{\alpha,k}[\phi](x) := \int_{\partial D} G^{\alpha,k} (x,y) \phi(y) \dx\sigma(y),\quad x\in \mathbb{R}^d.$$
Here, the space $H_{\textrm{loc}}^1(\R^d)$ consists of functions that are square integrable and with a square integrable weak first derivative on every compact subset of $\R^d$. Taking the trace on $\p D$, it is well-known that $\S_D^{\alpha,0} : L^2(\p D) \rightarrow H^1(\p D)$ is invertible if $\alpha \neq  0$ \cite{MaCMiPaP}. For low frequencies, i.e., as $k \to0$, we have the asymptotic expansion (see, for instance, \cite{MaCMiPaP})
\begin{equation}\label{eq:Sexp}
	\mathcal{S}_D^{\alpha,k} = \mathcal{S}_D^{\alpha,0} + O(k^2),
\end{equation}
valid uniformly for $|\alpha| > c > 0$.
\begin{defn}[Capacitance matrix]
	For $\alpha \neq 0$, the basis functions $\psi_i^\alpha$ and the capacitance coefficients $C_{ij}^\alpha$ are defined as
\begin{equation}\label{eq:psiC}
	\psi_i^\alpha = \left(\S_D^{\alpha,0}\right)^{-1}[\chi_{\p D_i}], \qquad C_{ij}^\alpha= -\int_{\p D_i} \psi_j^\alpha   \dx \sigma,
\end{equation}
for $i,j=1,\ldots,N$, where $\chi_{\p D_i}$ is the characteristic function of $\partial D_i$. The capacitance matrix $C^\alpha$ is defined as the matrix $C^\alpha = \left(C_{ij}^\alpha\right)$.
\end{defn}
As we shall see, the capacitance matrix provides, to leading order, an asymptotic approximation of the equation \eqref{eq:wave_transf} as $\delta \to 0$. The following results hold.
\begin{lem}[\cite{ammari2021functional}] \label{lem:herm}
	The capacitance matrix $C^{\alpha}$ is Hermitian.
\end{lem}
\begin{lem}
	\label{cap}
	For all $\alpha \in Y^*$, we have $C^{-\alpha}=\overline{C^{\alpha}}=(C^\alpha)^\top$, where the superscript $\top$ denotes the transpose. 
\end{lem}
\begin{proof}
	The identity follows from the fact that $G^{-\alpha,k}=\overline{G^{\alpha,k}}$.
\end{proof}
\subsection{Time-modulated subwavelength resonators}
	\label{setting}
We seek solutions to  (\ref{eq:wave_transf}) with modulations given by  (\ref{modulation}). Since $e^{-\mathrm{i}\omega t}u(x,t)$ is a $T$-periodic function of $t$, we can write its Fourier series as
$$u(x,t)= e^{\iu \omega t}\sum_{n = -\infty}^\infty v_n(x)e^{\iu n\Omega t}.$$
In the frequency domain, we then have from (\ref{eq:wave_transf}) the following equation, for $n\in \Z$:
\begin{equation} \label{eq:freq}
	\left\{
	\begin{array} {ll}
		\ds \Delta {v_n}+ \frac{\rho_0(\omega+n\Omega)^2}{\kappa_0} {v_n}  = 0 & \text{in } Y \setminus \overline{D}, \\[0.3em]
		\ds \Delta v_{i,n}^* +\frac{\rho_r(\omega+n\Omega)^2}{\kappa_r} v_{i,n}^{**}  = 0 & \text{in } D_i, \\
		\nm
		\ds  {v_n}|_{+} -{v_n}|_{-}  = 0  & \text{on } \partial D, \\
		\nm
		\ds  \delta \frac{\partial {v_n}}{\partial \nu} \bigg|_{+} - \frac{\partial v_{i,n}^* }{\partial \nu} \bigg|_{-} = 0 & \text{on } \partial D_i, \\[0.3em]
		v_n(x)e^{\iu \alpha\cdot x} \text{ is $\Lambda$-periodic in $x$}.
	\end{array}
	\right.
\end{equation}
Here, $v_{i,n}^*(x)$ and $v_{i,n}^{**}(x)$ are defined through the convolutions
$$v_{i,n}^*(x) = \sum_{m = -\infty}^\infty r_{i,m} v_{n-m}(x), \quad  v_{i,n}^{**}(x) = \frac{1}{\omega+n\Omega}\sum_{m = -\infty}^\infty k_{i,m}\big(\omega+(n-m)\Omega\big)v_{n-m}(x),$$
where $r_{i,m}$ and $k_{i,m}$ are the Fourier series coefficients of $1/\rho_i$ and $1/\kappa_i$, respectively:
$$\frac{1}{\rho_i(t)} = \sum_{n = -\infty}^\infty r_{i,n} e^{\iu n \Omega t}, \quad \frac{1}{\kappa_i(t)} = \sum_{n = -\infty}^\infty k_{i,n} e^{\iu n \Omega t}.$$
We can assume that the solution is normalized as $\|v_0\|_{H^1(Y)} = 1$. Since $u$ is continuously differentiable in $t$, we then have as $n\to \infty$,
\begin{equation} \label{eq:reg_v}
	\|v_n\|_{H^1(Y)} = o\left(\frac{1}{n}\right).
\end{equation}  
We will consider the case when the modulation of $\rho$ and $\kappa$ consist of a finite Fourier series with a large number of nonzero Fourier coefficients: 
$$\frac{1}{\rho_i(t)} = \sum_{n = -M}^M r_{i,n} e^{\iu n \Omega t}, \qquad \frac{1}{\kappa_i(t)} = \sum_{n = -M}^M k_{i,n} e^{\iu n \Omega t},$$
for some $M\in \N$ satisfying
$$M = O\left(\delta^{-\gamma/2}\right),$$
for some $0 < \gamma < 1$. We seek subwavelength quasifrequencies $\omega$ of the wave equation \eqref{eq:wave_transf} in the sense of the following definition introduced in \cite{ammari2020time}.
 \begin{defn}[Subwavelength quasifrequency] \label{def:sub}
 		A quasifrequency $\omega = \omega(\delta) \in Y^*_t$ of \eqref{eq:wave_transf} is said to be a \emph{subwavelength quasifrequency} if there is a corresponding Bloch solution $u(x,t)$, depending continuously on $\delta$, which can be written as
 		$$u(x,t)= e^{\iu \omega t}\sum_{n = -\infty}^\infty v_n(x)e^{\iu n\Omega t},$$
 		where 
 		$$\omega \rightarrow 0 \ \text{and} \ M\Omega \rightarrow 0 \ \text{as} \ \delta \to 0,$$
 		for some integer-valued function $M=M(\delta)$ such that, as $\delta \to 0$, we have
 		$$\sum_{n = -\infty}^\infty \|v_n\|_{L^2(Y)} = \sum_{n = -M}^M \|v_n\|_{L^2(Y)} + o(1).$$
 	\end{defn}
In particular, we assume that the subwavelength quasifrequency $\omega$ and the frequency of modulation $\Omega$ have the same order:
$$\omega = O\left(\delta^{1/2}\right).$$ 	
%

The following is a capacitance matrix characterization of the band structure of time-dependent periodic systems of subwavelength resonators. 
\begin{thm}[\cite{ammari2020time}] \label{thm:pre}
		As $\delta \to 0$, the subwavelength quasifrequencies of the wave equation (\ref{eq:wave_transf}) are, to leading order, given by the quasifrequencies of the system of ODEs:
		\begin{equation}
		\label{Hill}
			\frac{\dx^2\phi}{\dx t^2}(t)+M^\alpha(t)\phi(t)=0,	
		\end{equation}
		where $M^\alpha$ is the matrix defined as
		\begin{equation*}
			M^\alpha(t)=\frac{\delta\kappa_r}{\rho_r}W_1(t)C^\alpha W_2(t)+W_3(t)
		\end{equation*}
		with $W_1,W_2$ and $W_3$ being the diagonal matrices with diagonal entries
		\begin{equation*}
			(W_1)_{ii}=\frac{\sqrt{\kappa_i}\rho_i}{\lvert D_i\rvert},\quad (W_2)_{ii}=\frac{\sqrt{\kappa_i}}{\rho_i},\quad (W_3)_{ii}=\frac{\sqrt{\kappa_i}}{2}\frac{\dx}{\dx\text{t}}\frac{\dx\kappa_i/\dx t}{\kappa_i^{3/2}}.
		\end{equation*}
\end{thm}
 
\begin{rmk}
	\Cref{thm:pre} provides an asymptotic approximation, namely \eqref{Hill}, of the original wave equation \eqref{eq:wave_transf}, valid in the high-contrast regime $\delta \to 0$. In the following, we shall only consider reciprocity of the approximating equation \eqref{Hill}. If we can prove that \eqref{Hill} has broken reciprocity, it follows that  equation \eqref{eq:wave_transf} has broken reciprocity for small enough $\delta$.
\end{rmk}
\section{Preservation of the reciprocity property despite time-modulations}
In this section, we give some sufficient time-modulation conditions for the preservation of the reciprocity property.
\subsection{Reciprocity preserved when $N=1,2$}
We first prove that if the number of resonators in the unit cell is one or two then the reciprocity of \eqref{Hill} is preserved. 
\begin{thm}
		If the number of resonators in the unit cell are less than $3$ (i.e., $N=1,2$), then the reciprocity of \eqref{Hill} is always preserved.
\end{thm}
\begin{proof}
	For $N=1$, $C^\alpha = C^{-\alpha}\in\mathbb{R}$. Hence the quasifrequencies are given by exactly the same equations. For $N=2$, we consider the general form of $M^\alpha(t)$ as in (\ref{Hill}):
	\begin{equation*}
		M^{\alpha}(t)=kW_1C^\alpha W_2 + W_3,
	\end{equation*}
  where $k$ is a real number and $W_1,W_2 \text{ and }W_3$ are diagonal. In view of \Cref{lem:herm}, we write the capacitance matrix as
  \begin{equation*}
  C^\alpha = \begin{pmatrix}
  	a &c\\
  	\overline{c} &b
  \end{pmatrix},
  \end{equation*} where $a,b \in \mathbb{R}$, and let
  \begin{equation*}
  	S = \begin{pmatrix}
  		\frac{(\overline{c})^2}{\lvert c\rvert ^2} & 0\\
 	0 & 1
  	\end{pmatrix},\ \ S^{-1}=S^{*}=	\begin{pmatrix}
			\frac{(c)^2}{\lvert c\rvert ^2} & 0\\
 		0 & 1
		\end{pmatrix}.
  \end{equation*} 
From Lemma \ref{cap}, we have
\begin{equation*}
\begin{split}
		S C^{\alpha}(t) S^*&= 
		\begin{pmatrix}
			\frac{(\overline{c})^2}{\lvert c\rvert ^2} & 0\\
 		0 & 1
		\end{pmatrix}	
		\begin{pmatrix}
			a & c\\
			\overline{c} &b
		\end{pmatrix}
		\begin{pmatrix}
			\frac{(c)^2}{\lvert c\rvert ^2} & 0\\
 		0 & 1
		\end{pmatrix}
		= \begin{pmatrix}
			a & \overline{c}\\
			{c} &b
		\end{pmatrix} = C^{-\alpha}.
\end{split}
\end{equation*}
Since diagonal matrices commute with each other, we have $SM^\alpha S^*=M^{-\alpha}$. As $S$ is time independent, we obtain that
\begin{equation*}
	M^{-\alpha}\phi =\frac{\dx\phi}{\dx t} \iff S M^{\alpha}S^*\phi=\frac{\dx\phi}{\dx t} \iff M^\alpha (S^*\phi)=\frac{\dx(S^*\phi)}{\dx t}.
\end{equation*}
This means that $\phi$ is a solution to (\ref{Hill}) with ${\alpha} \in Y^*$ if and only if $S\phi$ is that with ${-\alpha}$, having the same quasifrequency. 
\end{proof}
\subsection{Time-reversal symmetry preserves reciprocity}
When there are more than two resonators inside the unit cell, the reciprocity can be broken, as seen in the next section. Nevertheless, under the condition of time reversal symmetry, one can prove that the reciprocity is always preserved.
\begin{prop}
\label{reciprocallemma}
	If $\nu$ is a quasifrequency to the equation $M^{\alpha}(t)\phi(t)+(\dx^2\phi/\dx t^2)(t)=0$, then $-\overline{\nu}$ is a quasifrequency to the equation $M^{-\alpha}(t)\phi(t)+(\dx^2\phi/\dx t^2)(t)=0$. In particular, the real parts of the quasifrequency of the two equations differ in parity and the imaginary parts are the same. 
\end{prop}
\begin{proof}
	Let $\nu$ be a quasifrequency associated with  $M^{\alpha}(t)\phi(t) + (\dx^2\phi/\dx t^2)(t)=0$. This means that there is a solution $\phi$ to this system of ODEs such that $\phi(t+T)=e^{\iu \nu T}\phi(t)$. Since $M^{\alpha}(t)=\overline{M^{-\alpha}(t)}$, $\overline{\phi(t)}$ is then a solution to  $M^{-\alpha}(t)\phi(t)+(\dx^2\phi/\dx t^2)(t)=0$, having the quasifrequency $\overline{e^{\iu \nu T}}=e^{-\iu \bar{\nu} T}$.
\end{proof}
\begin{rmk}
	By Proposition \ref{reciprocallemma}, for the purpose of reciprocity, we should compare the real parts of the quasifrequencies associated with $\alpha$ and $-\alpha$. 
\end{rmk}
\begin{thm}
\label{timereversalthm}
	Let $t_0\in \mathbb{R}$ denote some initial time. If the time-modulation is time reversal symmetric, i.e., $M^{\alpha}(t)=M^{\alpha}(t_0-t)$ for all time $t\in \mathbb{R}$, then the reciprocity of \eqref{Hill} is preserved. In particular, without time-modulation, reciprocity of \eqref{Hill} is always preserved.
\end{thm}
\begin{proof}
	Pick a solution $\phi$ to \eqref{Hill} associated to the quasifrequency $\nu$, and define $\psi(t) = \overline{\phi(t_0-t)}$. Then $\psi$ is $\nu$-quasiperiodic.  Moreover, since $M^{\alpha}(t)\phi(t) + (\dx^2\phi/\dx t^2)(t)=0$ we find that $\psi$ satisfies
	$$\overline{M^{\alpha}(t_0-t)}\psi(t) + \frac{\dx^2\psi}{\dx t^2}(t)=0.$$
	Since $\overline{M^{\alpha}(t_0-t)} = M^{-\alpha}(t)$, we conclude that $\psi$ solves the reciprocal equation corresponding to the same quasifrequencies. We conclude that the sets of quasifrequencies at $\alpha$ and at $-\alpha$ coincide.
\end{proof}
\begin{rmk}
	The statements and arguments of both \Cref{reciprocallemma} and  \Cref{timereversalthm} easily generalize to the wave equation \eqref{eq:wave_transf}.
\end{rmk}
\begin{corollary}
\label{commutationthm}
	Let $\mathcal{R}$ denote the diagonal matrix $\mathrm{diag}(\rho_1, \ldots, \rho_N)$. If $\kappa_i$  is constant for all $i$, while $\mathcal{R}$ and $C^{\alpha}$ commute, then the reciprocity property is preserved.  
\end{corollary}
\begin{proof}
	Under the above assumption, $M^{\alpha}(t)=KC^{\alpha}$, where $K$ is some constant matrix and $C^{\alpha}$ is the capacitance matrix. By Theorem \ref{timereversalthm}, the reciprocity is preserved.
\end{proof}

\section{Asymptotic analysis of the non-reciprocity property}
When $N\geq 3$, reciprocity can be broken with time-modulation in $\rho(t)$ and $\kappa(t)$. In this section, in order to describe this non-reciprocity, we study the case of weak time-modulation where $M^\alpha$ is the sum of a constant
matrix and a small periodic perturbation. In other words, we assume that $M^\alpha(t)$ is an analytic function of $\varepsilon$ at $\varepsilon=0$ and can be written as 
$$M^\alpha(t) = M^\alpha_0 + \varepsilon M_1^\alpha(t) + \ldots + \varepsilon^n M_n^\alpha(t) + \ldots,$$
where $\varepsilon > 0$ is some small parameter describing the amplitude of the time-modulation and $ M^\alpha_0$ corresponds to the unmodulated case. Moreover, we assume that the above series converges for $|\varepsilon| < r_0$, where $r_0$ is independent of $t$. This holds true since the modulations in $\rho$ and $\kappa$ are with finite Fourier coefficients. 

To tackle the problem of non-reciprocity, we will use the asymptotic Floquet analysis developed in \cite{TheaThesis}, which is a combination of perturbation analysis and Floquet theory; see also \cite{Yakubovich}. Starting with the second-order ODE \eqref{Hill}, we can rewrite it into
\begin{equation}\label{eq:1dsys}
	\frac{\dx y}{\dx t}(t) = \widetilde{A}(t)y(t), \qquad \widetilde{A}=\begin{pmatrix}
	0 & \text{Id}\\
	-M^\alpha(t) & 0
\end{pmatrix}.
\end{equation}
We aim to give an asymptotic analysis of the quasifrequencies associated with $\alpha$ and $-\alpha$ in terms of $\varepsilon$. By Floquet theory as in Section \ref{Floquettheory}, we have $X(T)=e^{FT}$ so that the Floquet exponents are given by the eigenvalues of $F$. Our asymptotic analysis amounts to explicitly expand the Floquet matrix $F$, and then to apply eigenvalue perturbation theory to compute the quasifrequencies.

As an illustrative example, we will often consider the case where $\kappa $ is constant and $\rho$ is given by
\begin{equation} \label{example}
	\rho_i(t)= \frac{1}{1+\varepsilon \text{cos}(\Omega t + \phi_i)}, \quad i=1, \ldots, N, 
\end{equation}
where $\Omega$ is the frequency of the modulation and $\phi_i$ is a phase shift between the resonators.
\subsection{Floquet matrix elements}
In this section, we describe the asymptotic Floquet analysis in a general setting. We consider  the perturbed system of linear ODEs:
\begin{equation}
\label{varepsilonODE}
	\frac{\dx y}{\dx t}(t) = A_\varepsilon(t)y(t), 
\end{equation}
	and assume that the  $T$-periodic continuous matrix $A_\varepsilon(t)$ is an analytic function of $\varepsilon$ at $\varepsilon=0$ and has the following expansion: 
	\begin{equation} \label{expA}
	A_\varepsilon(t)=A_0+\varepsilon A_1(t) +\varepsilon^2 A_2(t) + 
\ldots + \varepsilon^n A_n(t) + \ldots,
	\end{equation}
	as $\varepsilon \rightarrow 0$. By Floquet's theorem (Theorem \ref{floquetthm}), the fundamental solution of \eqref{varepsilonODE} can be written as 
	\begin{equation}
		\label{floquetsolution}
		X_\varepsilon(t)=P_\varepsilon(t)e^{F_\varepsilon t}	,
	\end{equation}
	where $P_\varepsilon(t)$ is $T$-periodic with $P_\varepsilon(0)= \mathrm{Id}$ and $F_\varepsilon$ is constant in time. Crucially, we assume that	\begin{itemize}
		\item[(i)] $A_0$ is constant in time and diagonal;
		\item[(ii)] $F_0$ has no distinct eigenvalues which are congruent modulo $\tfrac{2\pi \iu}{T}$;
		\item[(iii)] the series in (\ref{expA}) is convergent for $|\varepsilon| < r_0$, where $r_0$ is independent of $t$.
	\end{itemize}
	Under these assumptions, it follows from \cite{Yakubovich} that the matrices $P_\varepsilon$ and $F_\varepsilon$ are analytic functions of $\varepsilon$ at $\varepsilon=0$ and therefore, they can be expanded as follows:
	\begin{equation}
	P_\varepsilon(t)=P_0(t)+\varepsilon P_1(t) +\varepsilon^2 P_2(t) + 
\ldots \ \text{and}\ \ F_\varepsilon=F_0+\varepsilon F_1 +\varepsilon^2 F_2 + 
\ldots.
	\end{equation}
By inserting (\ref{floquetsolution}) into (\ref{varepsilonODE}), we derive the following systems of ODEs:
\begin{equation} \left\{
\label{inductive}
\begin{split}
	&\frac{\dx P_0}{\dx t}(t)=A_0P_0(t)-P_0(t)F_0, \\
	&\frac{\dx P_n}{\dx t}(t)=A_0P_n(t)-P_n(t)F_0 +\sum_{i=1}^n(A_i(t)P_{n-i}(t)-P_{n-i}(t)F_i)\ \ \text{for}\ n\geq 1,
\end{split}	\right.
\end{equation}
with the initial conditions $P_0(0)=$ Id and $P_n(0)=0$ if $n\geq 1$. 

We remark that both $A_0$ and $F_0$ correspond to the unperturbed band functions (i.e., those associated with the unmodulated periodic system $\varepsilon = 0$). Nevertheless, in order to satisfy assumption (ii) above, we choose $F_0$ so that all eigenvalues are inside the first Brillouin  zone. In other words, $\text{Im}(\sigma(F_0))\subset [-\Omega/2,\Omega/2),$ where $\sigma(F_0)$ denotes the set of eigenvalues of $F_0$ (for further intuition on this folding, we refer to \Cref{fig:oneDresonators}). Since $F_0$ is defined modulo $\frac{2\pi\iu}{T}$, such choice is always possible, and is described in the following result \cite{TheaThesis}. 
\begin{lem}
	$A_0-F_0$ takes the diagonal form $\frac{2\pi \mathrm{i}}{T}\mathrm{diag}(m_1,m_2,\ldots,m_N)$, where $m_i$ is the folding number of $(A_0)_{ii}$ as in Definition \ref{foldingnumber}.
\end{lem}

We now vectorize (\ref{inductive}) by first vectorizing $P_0$ in the basis $\{ E_{kj} \}_{\{k,j=1,\ldots,N\}}$, where the $kj$-th entry $E_{kj}$ is $1$ and $0$ otherwise. We denote the vectorized quantity by vect$(P_0)=p_0$ with $$p_0(0)= \text{vect} (\text{Id})=\sum_{k=1}^Ne_{(k-1)N+k},$$ where $e_{(k-1)N+k}$ is an $N^2\times 1$ vector with the $({(k-1)N+k})$-th entry being $1$ and $0$ otherwise. Equation (\ref{inductive}) reads
\begin{equation}
	\label{tensored} \left\{
	\begin{split}
		&\frac{\dx p_0}{\dx t} (t) = (1\otimes A_0 - F_0 \otimes 1)p_0,\\
		&\frac{\dx p_n}{\dx t}(t)=(1\otimes A_0 - F_0 \otimes 1)p_n+\sum_{i=1}^n(1\otimes A_i- F_i^\top \otimes 1)p_{n-1}\ \ \text{for}\ n\geq 1.
	\end{split} \right.
\end{equation}
Here,  we have used the following tensor notation:
\begin{equation*}
	1\otimes A = \left(\begin{smallmatrix}
		A  & & &\\
		& A & &\\
	    & & \sddots\\
		& & & A
	\end{smallmatrix}\right), \qquad B\otimes 1 = \left(\begin{smallmatrix}
		B_{NN} \, \text{Id} & \cdots & B_{1N} \, \text{Id} \\
   \vdots & \sddots & \vdots \\
   B_{N1} \, \text{Id} & \cdots & B_{NN} \, \text{Id}
	\end{smallmatrix}\right).	
\end{equation*}

The following result holds.
\begin{lem}
We have the following expansion for $p_0$:
	\begin{equation}
	\label{p0}
	p_0 = \sum_{k=1}^N\mathrm{exp}(\mathrm{i}\Omega m_k t)e_{(k-1)N+k}.
	\end{equation}
\end{lem}
\begin{proof}
	Firstly, we write
	\begin{equation*}
			1\otimes A_0  - F_0 \otimes 1  = \left(\begin{smallmatrix}
				A_0  & & & \\
				& A_0 & & \\
				&  & \sddots & \\
				&  & & A_0
			\end{smallmatrix}\right)-\left(\begin{smallmatrix}
				(F_0)_{11} \, \text{Id}  & & & \\
				& (F_0)_{22} \, \text{Id} & & \\
				&  & \sddots & \\
				&  & & (F_0)_{NN}  \, \text{Id}
			\end{smallmatrix}\right).
	\end{equation*}
	We have 
	\begin{equation*}
		\begin{split}
			p_0 & = \text{exp}\big((1\otimes A_0  - F_0 \otimes 1) t\big)\; \sum_{k=1}^Ne_{(k-1)N+k}\\
			& = \sum_{k=1}^N\text{exp}((A_0-F_0)_{kk}t)e_{(k-1)N+k} \\
			& = \sum_{k=1}^N\text{exp}(\mathrm{i}\Omega m_k t)e_{(k-1)N+k}. \qedhere
		\end{split}
	\end{equation*}
\end{proof}
\noindent
Furthermore, from (\ref{inductive}) for $n=1$, it follows that
\begin{equation}
\label{Fourieransatz}
\begin{split}
	\frac{\dx p_1}{\dx t}(t)& =(1\otimes A_0- F_0 \otimes 1)p_1+(1\otimes A_1- F_1^\top \otimes 1)p_{0}\\
	& = (1\otimes A_0- F_0 \otimes 1)p_1+(1\otimes A_1- F_1^\top \otimes 1)\; \sum_{k=1}^N\text{exp}(i\Omega m_k t)e_{(k-1)N+k}.
\end{split}
\end{equation}
We insert the following Fourier series expansions: $$p_1(t)=\sum_{m\in \mathbb{Z}} \text{exp}(\mathrm{i}\Omega mt)p_1^{(m)}$$ and $$A_1(t)=\sum_{m\in \mathbb{Z}}\text{exp}(\mathrm{i}\Omega mt)A_1^{(m)}$$ into (\ref{Fourieransatz}) to obtain that
\begin{equation}
\begin{split}
	&\sum_{m\in \mathbb{Z}}\mathrm{i}\Omega m \; \text{exp}(\mathrm{i}\Omega mt)p_1^{(m)}\\
	=& \sum_{m\in \mathbb{Z}}(1 \otimes A_0 - F_0 \otimes 1) \text{exp}(\mathrm{i}\Omega mt)p_1^{(m)}
	 +(1\otimes A_1 - F_1^\top \otimes 1)\sum_{k=1}^N\text{exp}(\mathrm{i}\Omega m_k t)e_{(k-1)N+k}. \\
	\end{split}
	\end{equation}
	Then, it follows that
	\begin{multline}
	\label{solp1}
	\ds	\sum_{m\in \mathbb{Z}}\left(\mathrm{i}\Omega m- (1\otimes A_0- F_0 \otimes 1)\right)\text{exp}(\mathrm{i}\Omega mt)p_1^{(m)} \\ = \left(\sum_{m\in \mathbb{Z}}\text{exp}(\mathrm{i}\Omega mt) 1 \otimes A_1^{(m)}  - F_1^\top \otimes 1\right)\sum_{k=1}^N\text{exp}(\mathrm{i}\Omega m_k t)e_{(k-1)N+k}.
	\end{multline}
By comparing the coefficients in (\ref{solp1}), we conclude that for $m\in\mathbb{Z}$:
\begin{equation}
\label{importantp1}
(\mathrm{i}\Omega m- (1 \otimes A_0 - F_0 \otimes  1))p_1^{(m)}=\sum_{k=1}^N\left(1 \otimes A_1^{(m-m_k)}\right) e_{(k-1)N+k}-\sum_{k=1}^N(\delta_{mm_k} F_1^\top \otimes 1)e_{(k-1)N+k},
\end{equation}
where $\delta_{mm_k}$ is the Kronecker symbol. 
\begin{lem}
\label{diagonalF1}
	For every $j=1,\ldots,N$, we have $(F_1)_{jj} = (A_1^{(0)})_{jj}$.
\end{lem}
\begin{proof}
	For $j=1,\ldots,N$, we consider  (\ref{importantp1}) with $m=m_j$. The $((j-1)N+j)$-th entry of the left-hand side is $0$. Multiplying by $e_{(j-1)N+j}^\top$ gives us the $((j-1)N+j)$-th entry of the right-hand side as well:
	\begin{multline*}
		{e^\top_{(j-1)N+j}}\left(  \sum_{k=1}^N  1 \otimes A_1^{(m_j-m_k)} e_{(k-1)N+k}-\sum_{k=1}^N(\delta_{m_jm_k}F_1^\top \otimes 1) e_{(k-1)N+k} \right)\\
		=\sum_{k=1}^N \delta_{jk} \left( \left(A_1^{(m_j-m_k)}\right)_{jk}-(F_1^\top)_{jk} \right).
	\end{multline*}
	This means precisely that $(A_1^{(0)})_{jj}=(F_1)_{jj}$.
\end{proof}
\begin{lem}
\label{double}
	If $(F_0)_{ll}=(F_0)_{jj}$ for some $l\neq j$, then $(F_1)_{jl}=(A_1^{(m_l-m_j)})_{jl}$.
\end{lem}
\begin{proof}
	Consider the equation (\ref{importantp1}) with $m=m_l$. If $(F_0)_{ll}=(F_0)_{jj}$ for some $l\neq j$, then the $((l-1)N+j)$-th entry of the left-hand side is $0$.  Multiplying by $e_{(l-1)N+j}^\top$ gives us the $((l-1)N+j)$-th entry of the right-hand side as well:
	\begin{equation*}
	\begin{split}
		& {e^\top_{(l-1)N+j}} \left( \sum_{k=1}^N  1 \otimes A_1^{(m_l-m_k)} e_{(k-1)N+k}-\sum_{k=1}^N(\delta_{m_lm_k}F_1^\top \otimes 1) e_{(k-1)N+k} \right)\\
		&\qquad = \sum_{k=1}^N \delta_{jk}\left(A_1^{(m_l-m_k)}\right)_{lk}-\sum_{k=1}^N\delta_{m_lm_k}(F_1^\top)_{jl}\delta_{kl}\\
		&\qquad = \left(A_1^{(m_l-m_j)}\right)_{lj}-(F_1)_{lj}. \qedhere
	\end{split}
	\end{equation*}
\end{proof}
\begin{rmk}
\Cref{double} does not provide the whole structure of $F_1$; for example, the entry $(F_1)_{lk}$ when $(F_1)_{ll}\neq (F_1)_{kk}$ is not provided. However, we will see that this is sufficient for the computation of the eigenvalue perturbation up to linear order in $\varepsilon$. As we will see, this is sufficient to demonstrate broken reciprocity.
\end{rmk}

\subsection{Asymptotic analysis of quasifrequency perturbations and reciprocity}Assume that $F=F_0+\varepsilon F_1+\varepsilon^2 F_2 + \ldots$ and $F_0$ is diagonal with respect to the basis vectors $w_1,\ldots,w_N$. 
\begin{defn}[Degenerate point] Suppose that $f_0$ is a multiple eigenvalue of $F_0$ of multiplicity $r \geq 2$. Then $f_0$ is called a {\em degenerate point}. If $f_0$ is simple, then we call it a {\em non-degenerate point}. \end{defn}
In the remainder of this paper, we will focus on the perturbation of degenerate points. Let $f_0$ be a degenerate point of multiplicity $r$ and let  $w_1,\ldots, w_r$ be its associated eigenvectors. Without loss of generality, we assume that $(F_0)_{ii} = f_0$ for $i=1,\ldots,r$, i.e., the diagonal entries of $A_0$ are permuted to make  the first $r$ diagonal entries of $F_0$ coincide. In this setting, there are standard expansions for the eigenvalue perturbation of $f_0$, which we outline in Appendix \ref{app:pert}.

For the purpose of reciprocity, we only need the upper left $r\times r$  block of $F_1$, which is given by Lemmas \ref{diagonalF1} and \ref{double}, to derive the first order perturbation at degenerate points or eigenvalues of $F_0$. The following result gives an asymptotic expansion of the quasifrequencies in terms of $\varepsilon$.
\begin{thm}
\label{order1}
Let $f_0$ be a degenerate point with multiplicity $r$. Then $F$ has associated eigenvalues given by 
$$f_0 + \varepsilon f_i + O(\varepsilon^2),$$
where $f_i$, for $i=1,\ldots,r$, are the eigenvalues of the $r\times r$ upper-left block of $F_1$, whose entries are given by
	\begin{equation} \label{eq417}
		(F_1)_{lk}=\left(A_1^{(m_l-m_k)}\right)_{lk} \ \text{for} \ l,k = 1,\ldots,r,
	\end{equation}
	where $m_l$ and $m_k$ denote the folding numbers of the $l$-th and $k$-th eigenvalues of $A_0$.
\end{thm}
\begin{proof}
Formula \eqref{eq417} follows from Lemmas \ref{diagonalF1} and \ref{double} together with  (\ref{effective}).
\end{proof}
Typically, it is sufficient to consider degenerate points of order $r=2$. Moreover, it is natural to assume that $A_1$ has no constant part; i.e., $A_1^{(0)} = 0$. In this setting, the $\varepsilon$-perturbations are given by the eigenvalues of the matrix
\begin{equation}
	\label{pvalue}
	\begin{pmatrix}
		0 & (F_1)_{12}\\
		(F_1)_{21} & 0	
	\end{pmatrix}.
\end{equation}
Thus, the eigenvalues $f$ of $F$ associated with $f_0$ are given by 
\begin{equation}\label{eq:pert}
	f = f_0 \pm \varepsilon\sqrt{(F_1)_{12}(F_1)_{21}} + O(\varepsilon^2).
\end{equation}

\begin{example}
\label{ourmodulation}
	Consider the same setting as the one in Section \ref{setting}. Let $\kappa_i$ be constant and let $N=3$. Define the time-modulation of $\rho_i,$ for $i=1,2,3,$  by
	(\ref{example}). 
	Then $M^\alpha(t)=K\rho(t)C^{\alpha}\rho^{-1}(t)$ and $K$ is a constant matrix. A direct computation shows that $M^\alpha(t)=K\widetilde{C}^\alpha(t)$, where $\widetilde{C}^\alpha_{ij}=C^{\alpha}_{ij}\rho_i\rho_j^{-1}$, for  $i,j=1,2,3$. By expanding  $\widetilde{C}^\alpha$, we obtain that
\begin{equation}
\label{Fourierexpansion}	
\widetilde{C}^\alpha_{ij}=C^{\alpha}_{ij}\rho_i\rho_j^{-1}=C^{\alpha}_{ij} \bigg(1+\varepsilon (-\text{cos}(\Omega t+\phi_i)+\text{cos}(\Omega t+\phi_j))+O(\varepsilon^2) \bigg). 
\end{equation}
We see that $M^\alpha(t)$ does not have a constant Fourier coefficient in the first $\varepsilon$-order. Therefore, non-degenerate points will depend quadratically on $\varepsilon$, whereas degenerate points will depend linearly on $\varepsilon$.
\end{example}

\begin{example}\label{ex}
Consider the same modulations as before with  phase shifts $\phi_1 = 0,\ \phi_2=\pi/2$ and $\phi_3 =\pi$. Suppose that the modulation frequency $\Omega$ is chosen such that the static system has a degenerate point, e.g., as in the case of a one-dimensional lattice with unit cell containing three resonators. The band structure of such structure is depicted in Figure \ref{fig:oneDresonators}, where we have a double degenerate point at $\alpha=\pm 2.23$ with different folding numbers: $m_1 = 0$ and $m_2 = 1$.
\begin{figure}[H]
	\begin{subfigure}[b]{0.5\linewidth}
		\begin{center}
      \includegraphics[height=7cm]{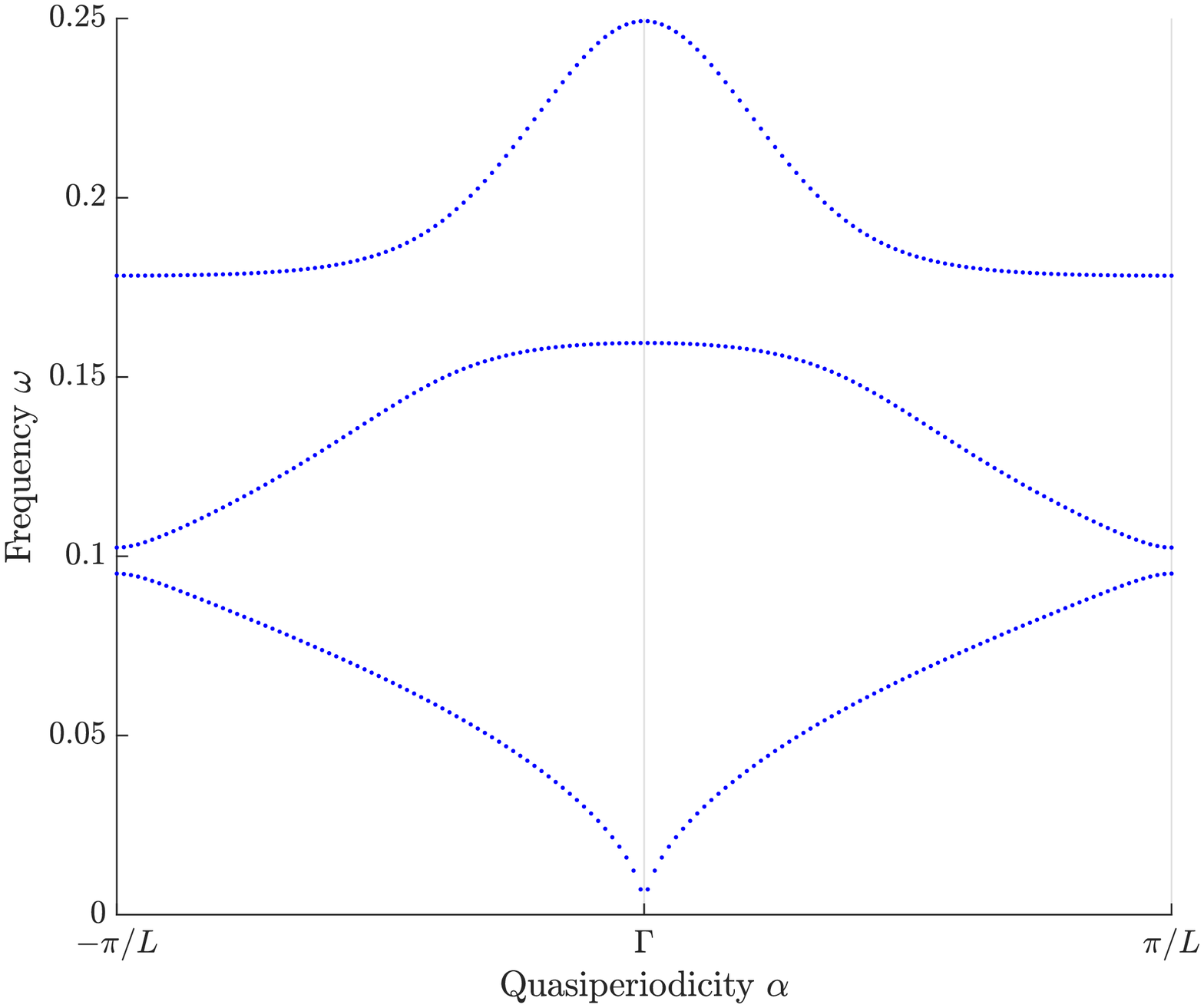}
		\caption{Unfolded band structure.}
		\end{center}	\end{subfigure}
	\begin{subfigure}[b]{0.5\linewidth}
	\begin{center}
\includegraphics[height=7cm]{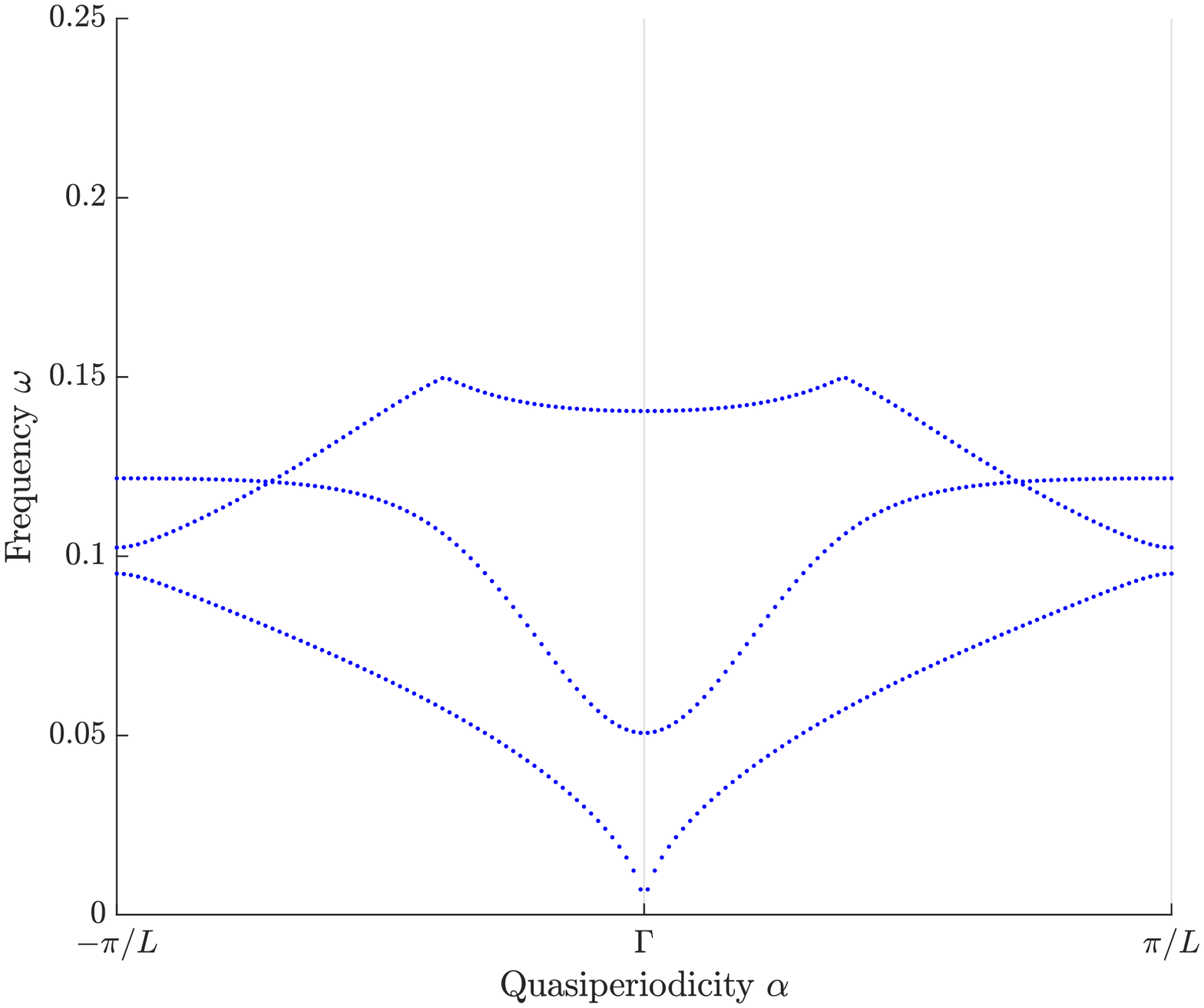}
	\caption{Folded band structure with $\Omega = 0.3$.}
	\end{center}	
	\end{subfigure}
	\caption{ Subwavelength band functions of the one-dimensional lattice of unit cells containing three resonators in the static (i.e., unmodulated) case. In (b) there is a degenerate point around $\alpha = 2.23$.}
	\label{fig:oneDresonators}
	\end{figure}
In Figure \ref{fig:oneDresonators_perturbation4figures}, we demonstrate the subwavelength band functions of \Cref{fig:oneDresonators} as $\varepsilon$ increases from $0$. Note that in the absence of time-modulation, the band functions are symmetric for opposite directions. The time-modulation  opens non-symmetric band gaps at degenerate points, as a consequence of breaking time-reversal symmetry. If the excitation frequency falls inside the band gap for only one propagation direction, wave transmission is prohibited in this direction but not in the opposite one.
\end{example}
\begin{figure}[p]
	\begin{subfigure}[b]{0.49\linewidth}
		\vspace{0pt}
		\begin{center}
	\includegraphics[width=1\linewidth]{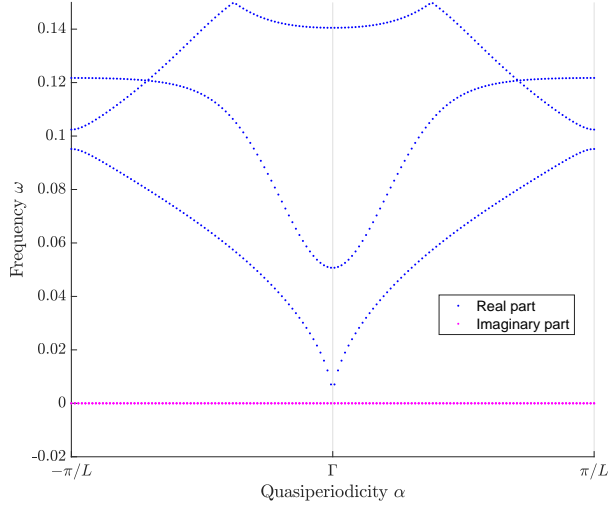}
		\end{center}
		\caption{Static ($\varepsilon=0$) band structure of the one-dimensional lattice of resonators, folded with $\Omega = 0.3$. \newline}
	\end{subfigure}\hfill
	\begin{subfigure}[b]{0.49\linewidth}
		\begin{center}
\includegraphics[width=1\linewidth]{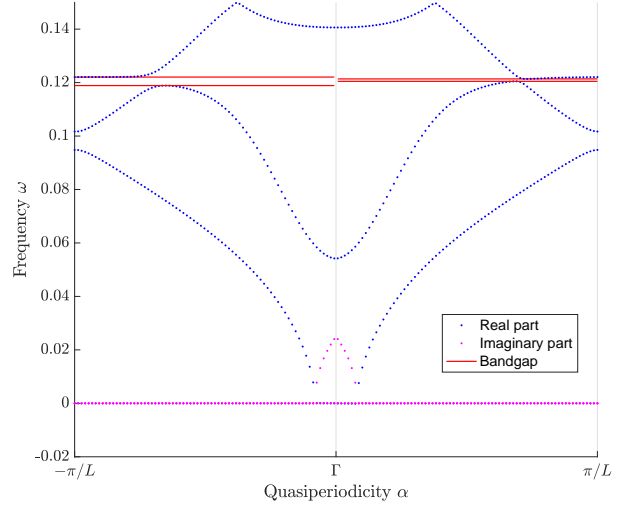}
		\end{center}
		\caption{Time-modulated ($\varepsilon=0.1$) band structure of the one-dimensional lattice of resonators, folded with  $\Omega = 0.3$.\newline}
	\end{subfigure}	

	\vspace{20pt}

	\begin{subfigure}[b]{0.48\linewidth}
	\begin{center}
	\includegraphics[width=1\linewidth]{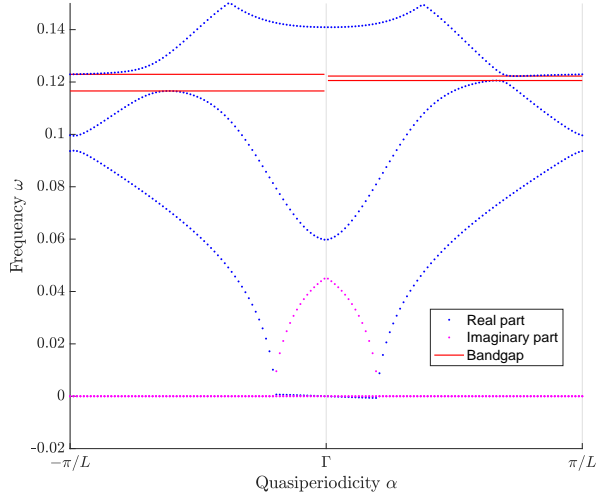}
	\end{center}
	\caption{Time-modulated ($\varepsilon=0.2$) band structure of the one-dimensional lattice of resonators, folded with $\Omega = 0.3$.}	
	\end{subfigure}\hfill
	\begin{subfigure}[b]{0.48\linewidth}
		\begin{center}
     \includegraphics[width=1\linewidth]{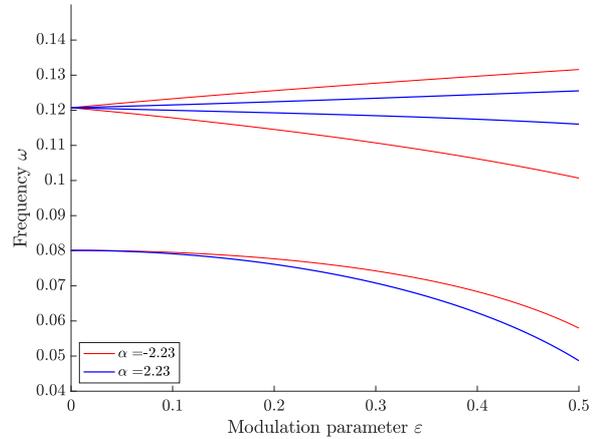}
		\end{center}
		\caption{$\Omega = 0.3$, $\alpha=\pm 2.23$ showing perturbations of degenerate points (depending linearly on $\varepsilon$), and non-degenerate points (depending quadratically on $\varepsilon$).}\label{fig:oneDresonators_perturbation}
	\end{subfigure}
	
	\caption{Band structure of one-dimensional trimers ($N=3$) of subwavelength resonators with time-modulation in $\rho$.  As $\varepsilon$ increases from $0$, the degenerate points open into asymmetric band gaps, indicated by solid red lines in (b) and (c). Waves with frequencies inside the disjunctive union of these band gaps can only propagate in one direction.}
	\label{fig:oneDresonators_perturbation4figures}
\end{figure}
Using formula \eqref{eq:pert}, we now seek to verify the observations made in \Cref{ex}. For the purpose of reciprocity, we would like to determine the values of $p_\alpha:=(F_1^\alpha)_{12}(F_1^\alpha)_{21}$ (where we have made the $\alpha$-dependence of $F_1$ explicit) and verify that $p_\alpha \neq p_{-\alpha}$. Recall the system of $2N$ linear ODEs given in \eqref{eq:1dsys}:
\begin{equation*}
	\frac{\dx y}{\dx t}(t) = \widetilde{A}(t)y(t),
\end{equation*}
with $\widetilde{A}=\widetilde{A}_0+\varepsilon \widetilde{A}_1+O(\varepsilon^2)$, where
\begin{equation*}
	\widetilde{A}_0=\begin{pmatrix}
		0 & \mathrm{Id}\\
		-M^\alpha_0 & 0
	\end{pmatrix}, \qquad
	\widetilde{A}_1=\begin{pmatrix}
		0 & 0\\
		-M^\alpha_1(t) & 0
	\end{pmatrix}.
\end{equation*}
We assume that $\widetilde{A}_0$ is diagonalizable and the matrices $S$ and $S^{-1}$ diagonalize $\widetilde{A}_0$ to $A_0$:
\begin{equation*}
	S\title{A}_0 S^{-1}=\widetilde{A}_0.
\end{equation*}
Then we define $A_1$ in the same way: $A_1=S\widetilde{A}_1S^{-1}$. Let now $\omega_0$ be a $2$-fold degenerate point. By Theorem \ref{order1}, we have
\begin{equation*}
	(F_1)_{12}=(A_1^{(m_1-m_2)})_{12},\qquad (F_1)_{21}=(A_1^{(m_2-m_1)})_{21}.
\end{equation*}
This implies the following result.
\begin{thm}
Let $\alpha\in Y^*$, and let $\omega_0^\alpha$ be a double degenerate point of $\widetilde{A}_{0,\alpha}$ (i.e. with $r=2$). Then the perturbed eigenvalue $\omega^\alpha$ of $\widetilde{A}_\alpha$ satisfies $\omega^\alpha = \omega_0^\alpha \pm \varepsilon r_\alpha + O(\varepsilon^2)$, where $r_\alpha:=\sqrt{p_\alpha}$ and
\begin{equation}
\label{pvaluecompute}
	p_\alpha = e^\top_1S_\alpha\widetilde{A}^{(m_1-m_2)}_{1,\alpha}S_\alpha^{-1} e_2e^\top_2S_\alpha\widetilde{A}_{1,\alpha}^{(m_2-m_1)}S_\alpha^{-1} e_1,
\end{equation}
where $S_\alpha$ diagonalizes $\widetilde{A}_0^\alpha$ and $e_1$ and $e_2$ denote the first and second standard basis vectors in $\mathbb{R}^N$.
\end{thm} 

	In general, if $N\geq 3$ we have, $r_\alpha\neq r_{-\alpha}$.  Consider a one-dimensional lattice with three resonators in the unit cell, i.e., the same setting as in the above example. 	In this setting, $(M^\alpha_1(t))_{ij}=KC^{\alpha}_{ij}(\text{cos}(\Omega t+\phi_i)-\text{cos}(\Omega t + \phi_j))$. Here, $K$ is some constant independent of $\alpha$ and $t$ and $C^{\alpha}$ is the capacitance matrix. Assume that at $\alpha=\alpha_{\mathrm{deg}}$ we have a double degeneracy with folding numbers $m_1 =0$ and $m_2=1$. The computation of $p_\alpha$ as in \eqref{pvaluecompute} boils down to the $-1$ and the $1$-st Fourier coefficients of $M^\alpha_1(t)$. In \Cref{tab:vals}, we present results obtained for the rates of the first order perturbations depending on $\Omega$ with phase shifts: $\phi_1=0,\phi_2=\pi/2$ and $\phi_3=\pi$, computed in two different ways. The first method makes use of the asymptotic formula \eqref{pvaluecompute} while the second one is an ``exact'' one and is based on the multipole method developed in \cite{ammari2020time}. The reciprocity is broken in this setting.
\label{reci}
	\begin{table}[H]
\begin{subfigure}[b]{0.45\linewidth}
\centering
\begin{tabular}{||c c c c||} 
\hline
 $\Omega$ & $\alpha_{\text{deg}}$ & $r_\alpha$ & $r_{-\alpha}$ \\ [0.5ex] 
\hline
\hline
 $0.2$ & $-2.35$ & $\pm0.0165$ & $\pm 0.0058$ \\ 
\hline
 $0.3$ & $-2.23$ & $\pm0.0271$ & $\pm 0.0075$ \\  
\hline
 $0.4$ & $-0.32$ & $\pm0.0253$ & $\pm 0.0228$ \\ 
\hline
\end{tabular}
\caption{Values computed using the asymptotic formula \eqref{pvaluecompute}.}
\end{subfigure}\hfill
		\begin{subfigure}[b]{0.45\linewidth}
\centering
\begin{tabular}{||c c c c||} 
\hline
 $\Omega$ & $\alpha_{\text{deg}}$ & $r_\alpha$ & $r_{-\alpha}$ \\ [0.5ex] 
\hline
\hline
 $0.2$ & $-2.35$ & $\pm0.0165$ & $\pm 0.0058$ \\ 
\hline
 $0.3$ & $-2.23$ & $\pm0.0270$ & $\pm 0.0074$ \\  
\hline
 $0.4$ & $-0.32$ & $\pm0.0252$ & $\pm 0.0227$ \\  
\hline
\end{tabular}	
\caption{Values computed using the multipole discretization method.}
\end{subfigure}
\caption{Comparison between the first-order rates $r_\alpha$ of the eigenvalue perturbation, computed using the asymptotic formula (a) and using the multipole discretisation method (b). Here, we simulate the same system of subwavelength resonators as in Figure \ref{fig:oneDresonators_perturbation}.} \label{tab:vals}
\end{table}

%

\begin{rmk}
	In \cite{TheaThesis}, closed-form formulas for the elements of the matrices $F_1$ and $F_2$ are derived. Together with the eigenvalue perturbation theory in Appendix \ref{app:pert}, this allows us to compute higher-order asymptotic expansions of the quasifrequencies. In particular, we emphasize that the perturbation will generically scale as $O(\varepsilon^2)$. Nevertheless, at the degenerate points (which are the starting points for asymmetric band gap opening) the perturbation scales as $O(\varepsilon)$. The different behaviour between degenerate and non-degenerate points is also apparent in \Cref{fig:oneDresonators_perturbation}.
\end{rmk}

\section{Non-reciprocal transmission in other structures}
In the previous sections, we have explained the fundamental reasons for the broken reciprocity and analysed the perturbation of the Floquet exponents asymptotically. In this section, we provide numerical examples of other structures with broken reciprocity. The following examples originate from those considered in \cite{ammari2020time}. 

\subsection{Square lattice}
We begin by considering resonators in a 2-dimensional square lattice defined through the lattice vectors
		\begin{equation}l_1 = \begin{pmatrix}1\\0\end{pmatrix}, \quad l_2 = \begin{pmatrix}0\\1\end{pmatrix}.\end{equation}
		The lattice and the corresponding Brillouin zone are illustrated in \Cref{fig:square_uni}. The symmetry points in $Y^*$ are given by $\Gamma = (0,0), \ \text{M} = (\pi,\pi)$ and $\text{X}=(\pi,0)$.
\begin{figure}[H]
	\begin{subfigure}[b]{0.5\linewidth}
				\centering
				\begin{tikzpicture}[scale=1.2]
					\begin{scope}[xshift=-5cm,scale=1]
						\pgfmathsetmacro{\r}{0.1pt} 
						\coordinate (a) at (1,0);		
						\coordinate (b) at (0,1);
						
						\draw[opacity=0.2] (0,0) -- (1,0);
						\draw[opacity=0.2] (0,0) -- (0,1);
						\draw[fill=lightgray] (0.76,0.65) circle(\r); 
						\draw[fill=lightgray] (0.24,0.65) circle(\r);
						\draw[fill=lightgray] (0.5,0.2) circle(\r);
						
						\begin{scope}[shift = (a)]					
							\draw[opacity=0.2] (1,0) -- (1,1);
							\draw[opacity=0.2] (0,0) -- (1,0);
							\draw[opacity=0.2] (0,0) -- (0,1);
						\draw[fill=lightgray] (0.76,0.65) circle(\r); 
						\draw[fill=lightgray] (0.24,0.65) circle(\r);
						\draw[fill=lightgray] (0.5,0.2) circle(\r);
						\end{scope}
						\begin{scope}[shift = (b)]
							\draw[opacity=0.2] (1,1) -- (0,1);
							\draw[opacity=0.2] (0,0) -- (1,0);
							\draw[opacity=0.2] (0,0) -- (0,1);
						\draw[fill=lightgray] (0.76,0.65) circle(\r); 
						\draw[fill=lightgray] (0.24,0.65) circle(\r);
						\draw[fill=lightgray] (0.5,0.2) circle(\r);
						\end{scope}
						\begin{scope}[shift = ($-1*(a)$)]
							\draw[opacity=0.2] (0,0) -- (1,0);
							\draw[opacity=0.2] (0,0) -- (0,1);
						\draw[fill=lightgray] (0.76,0.65) circle(\r); 
						\draw[fill=lightgray] (0.24,0.65) circle(\r);
						\draw[fill=lightgray] (0.5,0.2) circle(\r);
						\end{scope}
						\begin{scope}[shift = ($-1*(b)$)]
							\draw[opacity=0.2] (0,0) -- (1,0);
							\draw[opacity=0.2] (0,0) -- (0,1);
						\draw[fill=lightgray] (0.76,0.65) circle(\r); 
						\draw[fill=lightgray] (0.24,0.65) circle(\r);
						\draw[fill=lightgray] (0.5,0.2) circle(\r);
						\end{scope}
						\begin{scope}[shift = ($(a)+(b)$)]
							\draw[opacity=0.2] (1,0) -- (1,1) -- (0,1);
							\draw[opacity=0.2] (0,0) -- (1,0);
							\draw[opacity=0.2] (0,0) -- (0,1);
						\draw[fill=lightgray] (0.76,0.65) circle(\r); 
						\draw[fill=lightgray] (0.24,0.65) circle(\r);
						\draw[fill=lightgray] (0.5,0.2) circle(\r);
						\end{scope}
						\begin{scope}[shift = ($-1*(a)-(b)$)]
							\draw[opacity=0.2] (0,0) -- (1,0);
							\draw[opacity=0.2] (0,0) -- (0,1);
						\draw[fill=lightgray] (0.76,0.65) circle(\r); 
						\draw[fill=lightgray] (0.24,0.65) circle(\r);
						\draw[fill=lightgray] (0.5,0.2) circle(\r);
						\end{scope}
						\begin{scope}[shift = ($(a)-(b)$)]
							\draw[opacity=0.2] (1,0) -- (1,1);
							\draw[opacity=0.2] (0,0) -- (1,0);
							\draw[opacity=0.2] (0,0) -- (0,1);
						\draw[fill=lightgray] (0.76,0.65) circle(\r); 
						\draw[fill=lightgray] (0.24,0.65) circle(\r);
						\draw[fill=lightgray] (0.5,0.2) circle(\r);
						\end{scope}
						\begin{scope}[shift = ($-1*(a)+(b)$)]
							\draw[opacity=0.2] (1,1) -- (0,1);
							\draw[opacity=0.2] (0,0) -- (1,0);
							\draw[opacity=0.2] (0,0) -- (0,1);
						\draw[fill=lightgray] (0.76,0.65) circle(\r); 
						\draw[fill=lightgray] (0.24,0.65) circle(\r);
						\draw[fill=lightgray] (0.5,0.2) circle(\r);
						\end{scope}
						\begin{scope}[shift = ($2*(a)$)]
							\draw (0.5,0.5) node[rotate=0]{$\cdots$};
						\end{scope}
						\begin{scope}[shift = ($-2*(a)$)]
							\draw (0.5,0.5) node[rotate=0]{$\cdots$};
						\end{scope}
						\begin{scope}[shift = ($2*(b)$)]
							\draw (0.5,0.3) node[rotate=90]{$\cdots$};
						\end{scope}
						\begin{scope}[shift = ($-2*(b)$)]
							\draw (0.5,0.7) node[rotate=90]{$\cdots$};
						\end{scope}
					\end{scope}
				\end{tikzpicture}
				\caption{Circular resonators in square lattice.}
			\end{subfigure}
			\begin{subfigure}[b]{0.5\linewidth}
				\centering
				\begin{tikzpicture}[scale=3]	
					\coordinate (a) at ({1/sqrt(3)},1);	
					\coordinate (b) at ({1/sqrt(3)},-1);
					\coordinate (c) at ({2/sqrt(3)},0);
					\coordinate (M) at (0.5,0.5);
					\coordinate (M2) at (-0.5,0.5);
					\coordinate (M3) at (-0.5,-0.5);
					\coordinate (M4) at (0.5,-0.5);
					\coordinate (X) at (0.5,0);
					\coordinate (X2) at (-0.5,0);

					\draw[->,opacity=0.8] (0,0) -- (0.8,0);
					\draw[->,opacity=0.8] (0,0) -- (0,0.8);
					\draw[fill] (M) circle(1pt) node[yshift=8pt, xshift=-2pt]{M}; 	\draw[fill] (M3) circle(1pt) node[left]{$-$M}; 

					\draw[fill] (0,0) circle(1pt) node[left]{$\Gamma$}; 
					\draw[fill] (X) circle(1pt) node[below right]{X}; 
					\draw[fill] (X2) circle(1pt) node[left]{$-$X}; 

					\draw[thick, postaction={decorate}, decoration={markings, mark=at position 0.1 with {\arrow{>}}, markings, mark=at position 0.3 with {\arrow{>}}, markings,  mark=at position 0.4 with {\arrow{>}}, markings, mark=at position 0.6 with {\arrow{>}}, markings, mark=at position 0.7 with {\arrow{>}}, markings, mark=at position 0.9 with {\arrow{>}}}, color=red]
					(0,0) -- (M) -- (X) -- (0,0) -- (X2) -- (M3) -- (0,0);
					\draw[opacity=0.8] (M) -- (M2) -- (M3) -- (M4) -- cycle; 
				\end{tikzpicture}
				\vspace{15pt}
				\caption{Brillouin zone and the symmetry points $\Gamma$, $\mathrm{X}$ and $\mathrm{M}$.}
			\end{subfigure}
			\caption{Illustration of the square lattice and the corresponding Brillouin zone. The red path shows the points where the band functions are computed.} \label{fig:square_uni}
\end{figure}
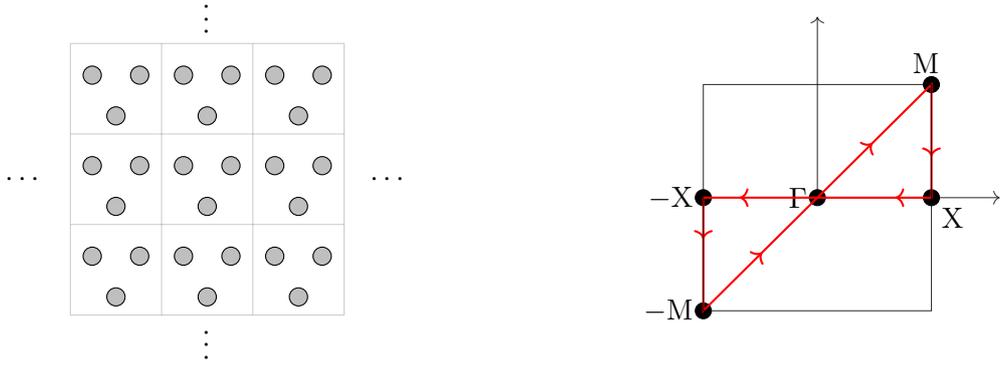
In Figure \ref{fig5}, we compute the band structure with modulation frequency $\Omega = 0.2$.

\begin{figure}[H]
	\begin{subfigure}[b]{0.4\linewidth}
		\vspace{0pt}
		\begin{center}
			\includegraphics[width=1\linewidth]{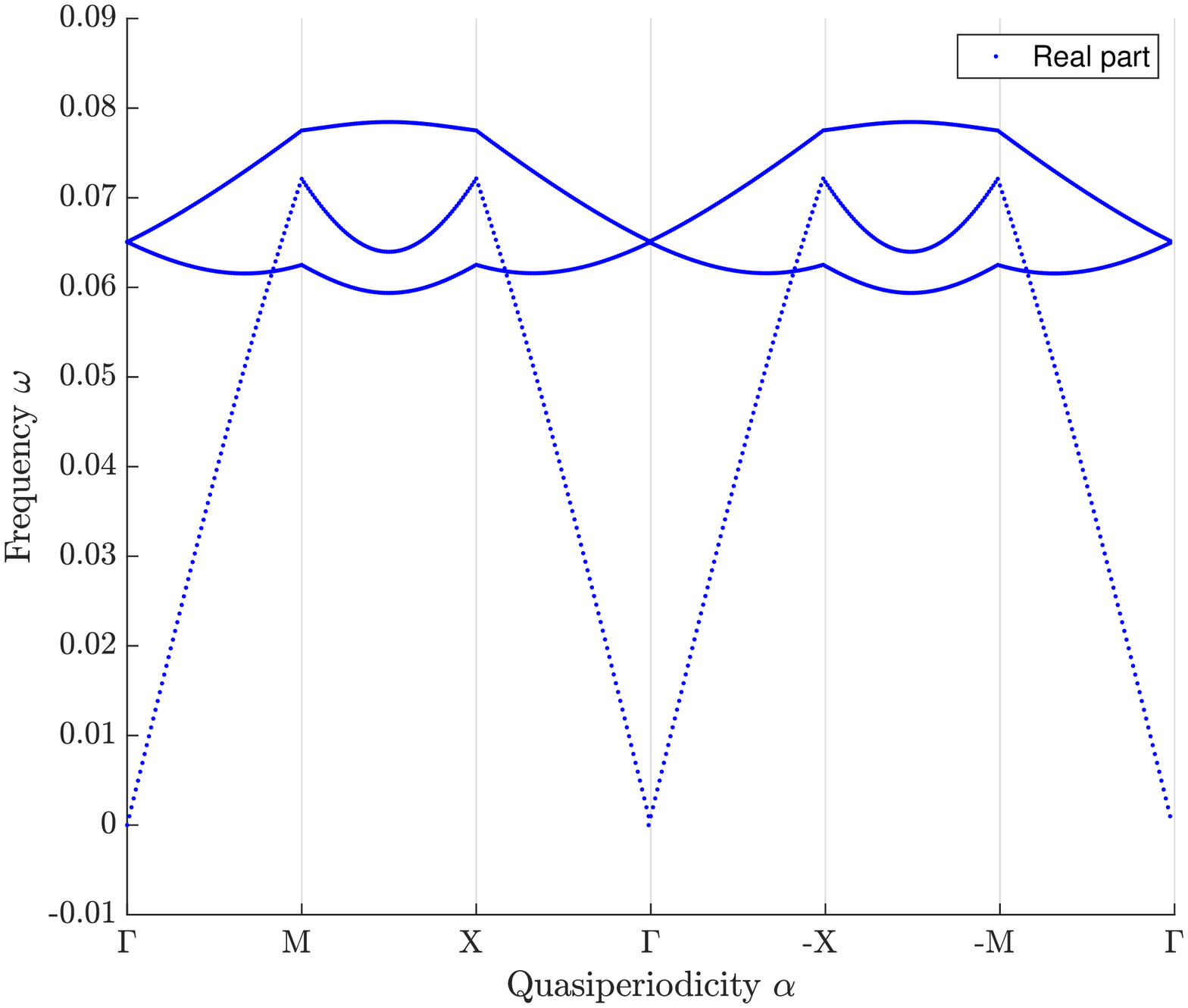}
		\end{center}
		\caption{ \centering
			Static case ($\varepsilon=0$).}
	\end{subfigure}
	\hspace{10pt}
	\begin{subfigure}[b]{0.4\linewidth}
			\vspace{0pt}
		\begin{center}
			\includegraphics[width=1\linewidth]{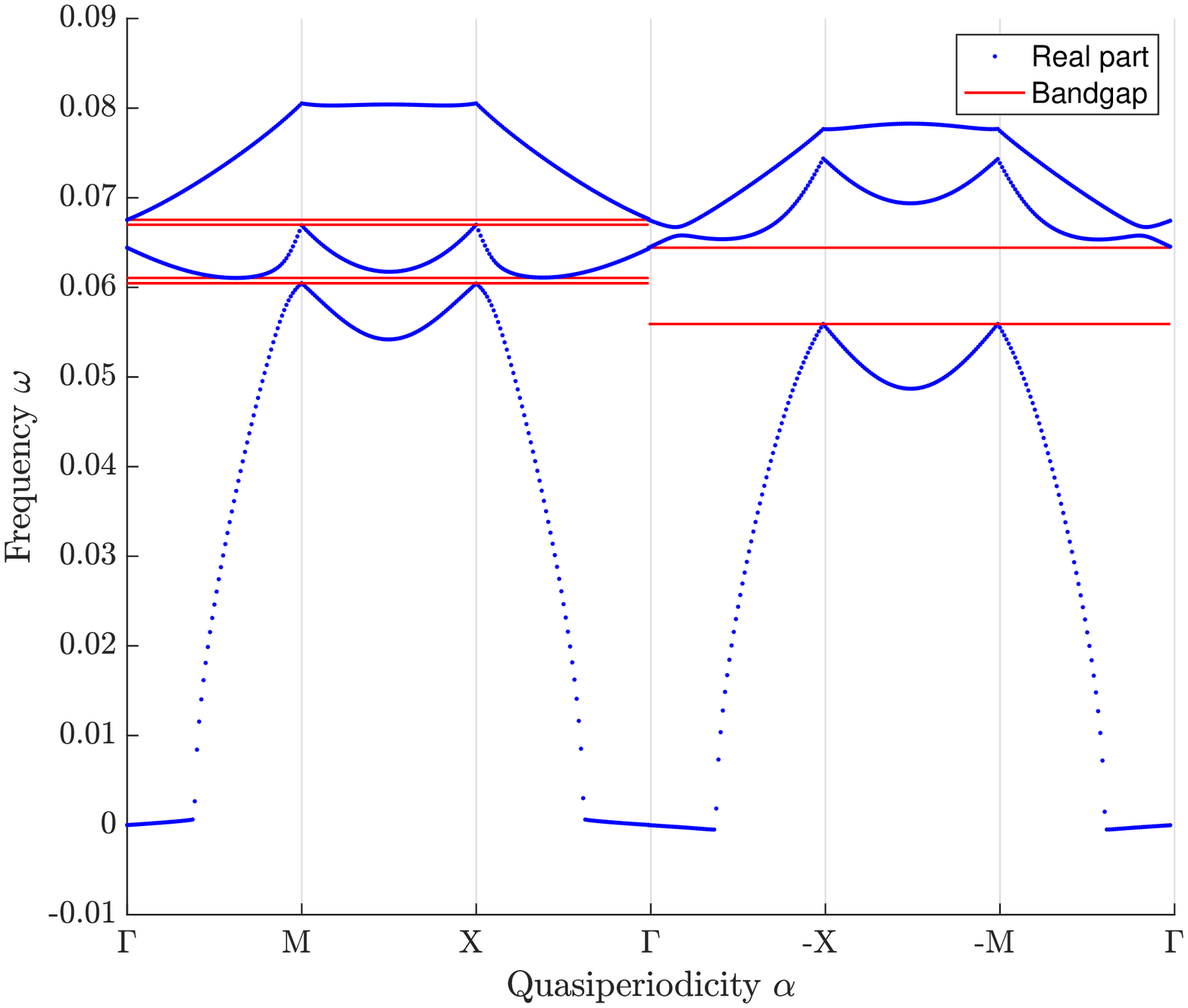}
		\end{center}
		\caption{\centering Modulated case with $\varepsilon=0.25$.}
	\end{subfigure}
	\centering
	\begin{subfigure}[b]{0.4\linewidth}
			\vspace{0pt}		
		\begin{center}
			\includegraphics[width=1\linewidth]{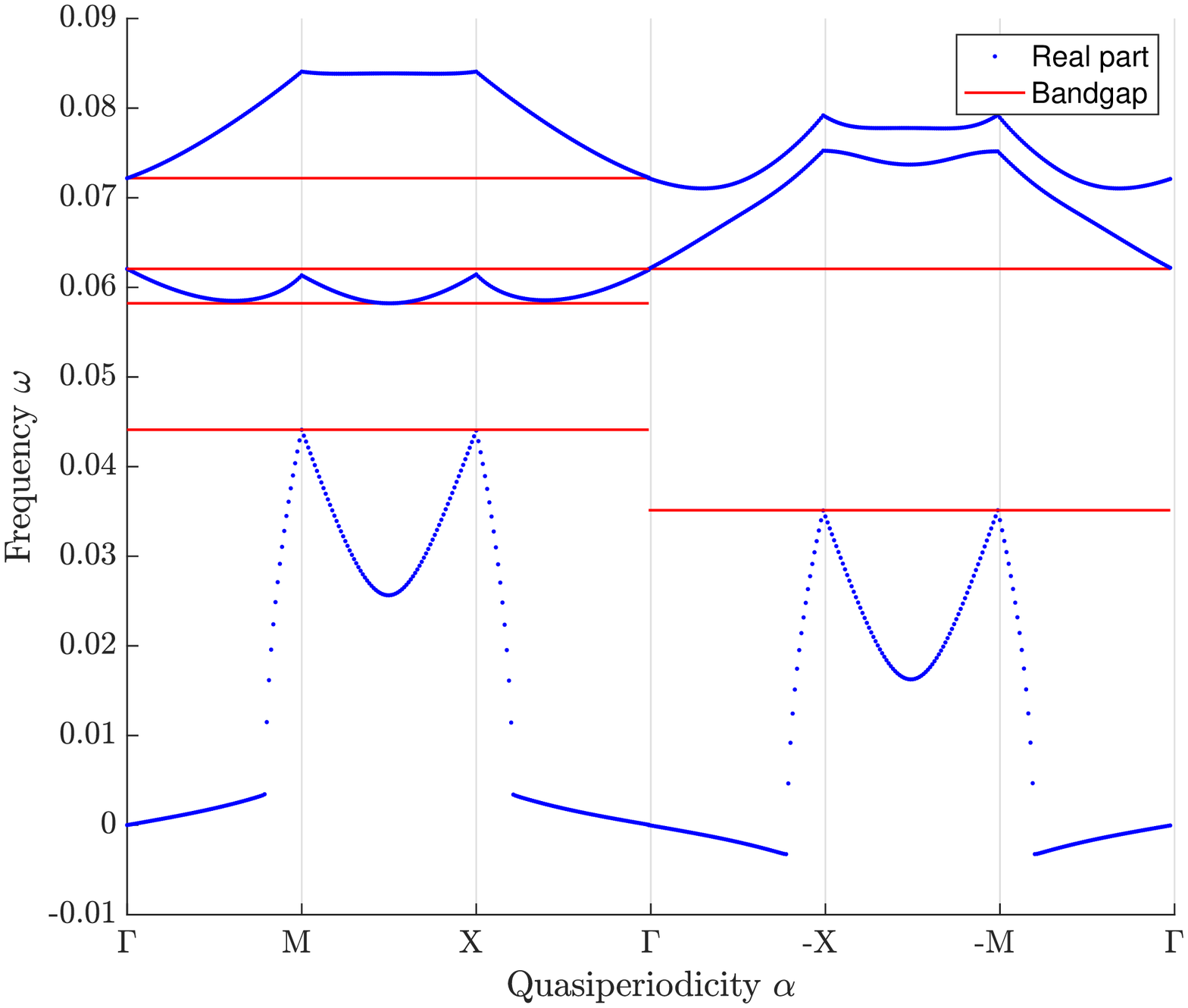}
		\end{center}
		\caption{\centering Modulated case with $\varepsilon=0.5$.}
	\end{subfigure}
	\hspace{10pt}
	\caption{Band structure of square lattice with three subwavelength resonators with modulation frequency $\Omega = 0.2$.} \label{fig5}
\end{figure}
\subsection{Honeycomb lattice}
First, we consider a honeycomb lattice of resonator trimers as illustrated in Figure \ref{fig:honeycomb}, where the unit cell now contains six resonators $D_i$ respectively centred at $c_i$, $i=1,..,6$:
\begin{align*}
	c_1 &= (1,0) + 3R(1,0), \quad c_2 = (1,0) + 3R\left(\cos\left(\tfrac{2\pi}{3}\right),\sin\left(\tfrac{2\pi}{3}\right)\right),   &c_3 = (1,0) + 3R\left(\cos\left(\tfrac{4\pi}{3}\right),\sin\left(\tfrac{4\pi}{3}\right)\right),\\[0.5em]
	c_4 &= (2,0) + 3R\left(\cos\left(\tfrac{\pi}{3}\right),\sin\left(\tfrac{\pi}{3}\right)\right), \qquad c_5 = (2,0) - 3R(1,0), 	 &c_6 = (2,0) + 3R\left(\cos\left(\tfrac{5\pi}{3}\right),\sin\left(\tfrac{5\pi}{3}\right)\right).
\end{align*}
We use the modulation given by $\kappa_i(t) = 1, \ i=1,\ldots,6$ and 
\begin{equation*}\rho_1(t) = \rho_4(t) = \frac{1}{1 + \varepsilon\cos(\Omega t)}, \quad \rho_2(t) = \rho_5(t) = \frac{1}{1 + \varepsilon\cos\left(\Omega t + \frac{2\pi}{3}\right)}, \quad \rho_3(t) = \rho_6(t) = \frac{1}{1 + \varepsilon\cos\left(\Omega t + \frac{4\pi}{3}\right)},\end{equation*}
for $0 \leq \varepsilon < 1$.

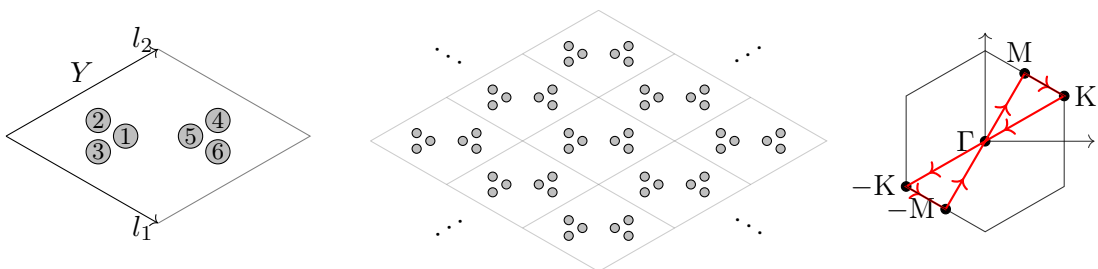
\begin{figure}[H]
\label{honeycomb}
	\begin{subfigure}[b]{0.33\linewidth}
		\centering
		\begin{tikzpicture}[scale=2]
				\pgfmathsetmacro{\r}{0.08pt}
				\pgfmathsetmacro{\rt}{0.12pt}
				\coordinate (a) at (1,{1/sqrt(3)});		
				\coordinate (b) at (1,{-1/sqrt(3)});	
				\coordinate (c) at (2,0);
				\coordinate (x1) at ({2/3},0);
				\coordinate (x2) at ({4/3},0);
				\draw[->] (0,0) -- (a) node[pos=0.9,xshift=0,yshift=7]{ $l_2$} node[pos=0.5,above]{$Y$};
				\draw[->] (0,0) -- (b) node[pos=0.9,xshift=0,yshift=-5]{ $l_1$};
				\draw[opacity=0.5] (a) -- (c) -- (b);
				\draw[fill=lightgray] (x1){} +(0:\rt) circle(\r) node{\footnotesize  1};
				\draw[fill=lightgray] (x1){} +(120:\rt) circle(\r)node{\footnotesize  2};
				\draw[fill=lightgray] (x1){} +(240:\rt) circle(\r) node{\footnotesize  3};
				\draw[fill=lightgray] (x2){} +(60:\rt) circle(\r) node{\footnotesize  4};
				\draw[fill=lightgray] (x2){} +(180:\rt) circle(\r) node{\footnotesize  5};
				\draw[fill=lightgray] (x2){} +(300:\rt) circle(\r) node{\footnotesize  6};
		\end{tikzpicture}
		\vspace{0.3cm}
		\caption{Hexagonal lattice unit cell $Y$ containing $6$ resonators.}
	\end{subfigure}
	\begin{subfigure}[b]{0.33\linewidth}
		\begin{tikzpicture}[scale=1]
			\begin{scope}[xshift=-5cm,scale=1]
				\pgfmathsetmacro{\r}{0.06pt}
				\pgfmathsetmacro{\rt}{0.12pt}
				\coordinate (a) at (1,{1/sqrt(3)});		
				\coordinate (b) at (1,{-1/sqrt(3)});
				
				\draw[opacity=0.2] (0,0) -- (a);
				\draw[opacity=0.2] (0,0) -- (b);
				\draw[fill=lightgray] ({2/3},0){} +(0:\rt) circle(\r);
				\draw[fill=lightgray] ({2/3},0){} +(120:\rt) circle(\r);
				\draw[fill=lightgray] ({2/3},0){} +(240:\rt) circle(\r);
				\draw[fill=lightgray] ({4/3},0){} +(60:\rt) circle(\r);
				\draw[fill=lightgray] ({4/3},0){} +(180:\rt) circle(\r);
				\draw[fill=lightgray] ({4/3},0){} +(300:\rt) circle(\r);
				
				\begin{scope}[shift = (a)]					
					\draw[opacity=0.2] (0,0) -- (1,{1/sqrt(3)});
					\draw[opacity=0.2] (0,0) -- (1,{-1/sqrt(3)});
					\draw[opacity=0.2] (1,{1/sqrt(3)}) -- (2,0);
					\draw[fill=lightgray] ({2/3},0){} +(0:\rt) circle(\r);
					\draw[fill=lightgray] ({2/3},0){} +(120:\rt) circle(\r);
					\draw[fill=lightgray] ({2/3},0){} +(240:\rt) circle(\r);
					\draw[fill=lightgray] ({4/3},0){} +(60:\rt) circle(\r);
					\draw[fill=lightgray] ({4/3},0){} +(180:\rt) circle(\r);
					\draw[fill=lightgray] ({4/3},0){} +(300:\rt) circle(\r);
				\end{scope}
				\begin{scope}[shift = (b)]
					
					\draw[opacity=0.2] (0,0) -- (1,{1/sqrt(3)});
					\draw[opacity=0.2] (0,0) -- (1,{-1/sqrt(3)});
					\draw[opacity=0.2] (2,0) -- (1,{-1/sqrt(3)});
					\draw[fill=lightgray] ({2/3},0){} +(0:\rt) circle(\r);
					\draw[fill=lightgray] ({2/3},0){} +(120:\rt) circle(\r);
					\draw[fill=lightgray] ({2/3},0){} +(240:\rt) circle(\r);
					\draw[fill=lightgray] ({4/3},0){} +(60:\rt) circle(\r);
					\draw[fill=lightgray] ({4/3},0){} +(180:\rt) circle(\r);
					\draw[fill=lightgray] ({4/3},0){} +(300:\rt) circle(\r);
				\end{scope}
				\begin{scope}[shift = ($-1*(a)$)]
					
					\draw[opacity=0.2] (0,0) -- (1,{1/sqrt(3)});
					\draw[opacity=0.2] (0,0) -- (1,{-1/sqrt(3)});
					\draw[fill=lightgray] ({2/3},0){} +(0:\rt) circle(\r);
					\draw[fill=lightgray] ({2/3},0){} +(120:\rt) circle(\r);
					\draw[fill=lightgray] ({2/3},0){} +(240:\rt) circle(\r);
					\draw[fill=lightgray] ({4/3},0){} +(60:\rt) circle(\r);
					\draw[fill=lightgray] ({4/3},0){} +(180:\rt) circle(\r);
					\draw[fill=lightgray] ({4/3},0){} +(300:\rt) circle(\r);
				\end{scope}
				\begin{scope}[shift = ($-1*(b)$)]
					
					\draw[opacity=0.2] (0,0) -- (1,{1/sqrt(3)});
					\draw[opacity=0.2] (0,0) -- (1,{-1/sqrt(3)});
					\draw[fill=lightgray] ({2/3},0){} +(0:\rt) circle(\r);
					\draw[fill=lightgray] ({2/3},0){} +(120:\rt) circle(\r);
					\draw[fill=lightgray] ({2/3},0){} +(240:\rt) circle(\r);
					\draw[fill=lightgray] ({4/3},0){} +(60:\rt) circle(\r);
					\draw[fill=lightgray] ({4/3},0){} +(180:\rt) circle(\r);
					\draw[fill=lightgray] ({4/3},0){} +(300:\rt) circle(\r);
				\end{scope}
				\begin{scope}[shift = ($(a)+(b)$)]
					
					\draw[opacity=0.2] (0,0) -- (1,{1/sqrt(3)});
					\draw[opacity=0.2] (0,0) -- (1,{-1/sqrt(3)});
					\draw[opacity=0.2] (1,{1/sqrt(3)}) -- (2,0) -- (1,{-1/sqrt(3)});
					\draw[fill=lightgray] ({2/3},0){} +(0:\rt) circle(\r);
					\draw[fill=lightgray] ({2/3},0){} +(120:\rt) circle(\r);
					\draw[fill=lightgray] ({2/3},0){} +(240:\rt) circle(\r);
					\draw[fill=lightgray] ({4/3},0){} +(60:\rt) circle(\r);
					\draw[fill=lightgray] ({4/3},0){} +(180:\rt) circle(\r);
					\draw[fill=lightgray] ({4/3},0){} +(300:\rt) circle(\r);
				\end{scope}
				\begin{scope}[shift = ($-1*(a)-(b)$)]
					
					\draw[opacity=0.2] (0,0) -- (1,{1/sqrt(3)});
					\draw[opacity=0.2] (0,0) -- (1,{-1/sqrt(3)});
					\draw[fill=lightgray] ({2/3},0){} +(0:\rt) circle(\r);
					\draw[fill=lightgray] ({2/3},0){} +(120:\rt) circle(\r);
					\draw[fill=lightgray] ({2/3},0){} +(240:\rt) circle(\r);
					\draw[fill=lightgray] ({4/3},0){} +(60:\rt) circle(\r);
					\draw[fill=lightgray] ({4/3},0){} +(180:\rt) circle(\r);
					\draw[fill=lightgray] ({4/3},0){} +(300:\rt) circle(\r);
				\end{scope}
				\begin{scope}[shift = ($(a)-(b)$)]
					
					\draw[opacity=0.2] (0,0) -- (1,{1/sqrt(3)});
					\draw[opacity=0.2] (0,0) -- (1,{-1/sqrt(3)});
					\draw[opacity=0.2] (1,{1/sqrt(3)}) -- (2,0);
					\draw[fill=lightgray] ({2/3},0){} +(0:\rt) circle(\r);
					\draw[fill=lightgray] ({2/3},0){} +(120:\rt) circle(\r);
					\draw[fill=lightgray] ({2/3},0){} +(240:\rt) circle(\r);
					\draw[fill=lightgray] ({4/3},0){} +(60:\rt) circle(\r);
					\draw[fill=lightgray] ({4/3},0){} +(180:\rt) circle(\r);
					\draw[fill=lightgray] ({4/3},0){} +(300:\rt) circle(\r);
				\end{scope}
				\begin{scope}[shift = ($-1*(a)+(b)$)]					
					\draw[opacity=0.2] (0,0) -- (1,{1/sqrt(3)});
					\draw[opacity=0.2] (0,0) -- (1,{-1/sqrt(3)});
					\draw[opacity=0.2] (2,0) -- (1,{-1/sqrt(3)});
					\draw[fill=lightgray] ({2/3},0){} +(0:\rt) circle(\r);
					\draw[fill=lightgray] ({2/3},0){} +(120:\rt) circle(\r);
					\draw[fill=lightgray] ({2/3},0){} +(240:\rt) circle(\r);
					\draw[fill=lightgray] ({4/3},0){} +(60:\rt) circle(\r);
					\draw[fill=lightgray] ({4/3},0){} +(180:\rt) circle(\r);
					\draw[fill=lightgray] ({4/3},0){} +(300:\rt) circle(\r);
				\end{scope}
				\begin{scope}[shift = ($2*(a)$)]
					\draw (1,0) node[rotate=30]{$\cdots$};
				\end{scope}
				\begin{scope}[shift = ($-2*(a)$)]
					\draw (1,0) node[rotate=210]{$\cdots$};
				\end{scope}
				\begin{scope}[shift = ($2*(b)$)]
					\draw (1,0) node[rotate=-30]{$\cdots$};
				\end{scope}
				\begin{scope}[shift = ($-2*(b)$)]
					\draw (1,0) node[rotate=150]{$\cdots$};
				\end{scope}
			\end{scope}
		\end{tikzpicture}
		
		\caption{Periodic system with trimers in a honeycomb lattice.}
	\end{subfigure}
	\begin{subfigure}[b]{0.33\linewidth}
		\centering
			\begin{tikzpicture}[scale=1.8]	
				\coordinate (a) at ({1/sqrt(3)},1);		
				\coordinate (b) at ({1/sqrt(3)},-1);
				\coordinate (c) at ({2/sqrt(3)},0);
				\coordinate (M) at ({0.5/sqrt(3)},0.5);
				\coordinate (M2) at ({-0.5/sqrt(3)},-0.5);
				\coordinate (Km) at ({-1/sqrt(3)},{-1/3});
				\coordinate (K1) at ({1/sqrt(3)},{1/3});
				\coordinate (K2) at ({1/sqrt(3)},{-1/3});
				\coordinate (K3) at (0,{-2/3});
				\coordinate (K4) at ({-1/sqrt(3)},{-1/3});
				\coordinate (K5) at ({-1/sqrt(3)},{1/3});
				\coordinate (K6) at (0,{2/3});
				
				\draw[->,opacity=0.8] (0,0) -- (0.8,0);
				\draw[->,opacity=0.8] (0,0) -- (0,0.8);
				\draw[fill] (M) circle(1pt) node[yshift=8pt, xshift=-2pt]{M}; 
				\draw[fill] (0,0) circle(1pt) node[left]{$\Gamma$}; 
				\draw[fill] (K1) circle(1pt) node[right]{K};

				\draw[fill] (M2) circle(1pt) node[left]{$-$M};
				\draw[fill] (Km) circle(1pt) node[left]{$-$K};

				\draw[thick,postaction={decorate}, decoration={markings, 
				 mark=at position 0.1 with {\arrow{>}}, markings,mark=at position 0.25 with {\arrow{>}},markings, mark=at position 0.45 with {\arrow{>}},markings,  mark=at position 0.65 with {\arrow{>}},markings, mark=at position 0.75 with {\arrow{>}}, markings, mark=at position 0.9 with {\arrow{>}}}, color=red]
				(0,0) -- (M) -- (K1) -- (0,0) -- (Km) -- (M2) -- (0,0);
	
				\draw[opacity=0.8] (K1) -- (K2) -- (K3) -- (K4) -- (K5) -- (K6) -- cycle; 
			\end{tikzpicture}
		\vspace{15pt}		
		\caption{Brillouin zone and the symmetry points $\Gamma$, $\mathrm{K}$ and $\mathrm{M}$.}
	\end{subfigure}
	\caption{Illustration of the honeycomb lattice and corresponding Brillouin zone. The red path shows the points where the band functions are computed.} \label{fig:honeycomb}
\end{figure}

\begin{figure}[tbh]
	\begin{subfigure}[b]{0.45\linewidth}
		\vspace{0pt}
		\begin{center}
		\includegraphics[width=1\linewidth]{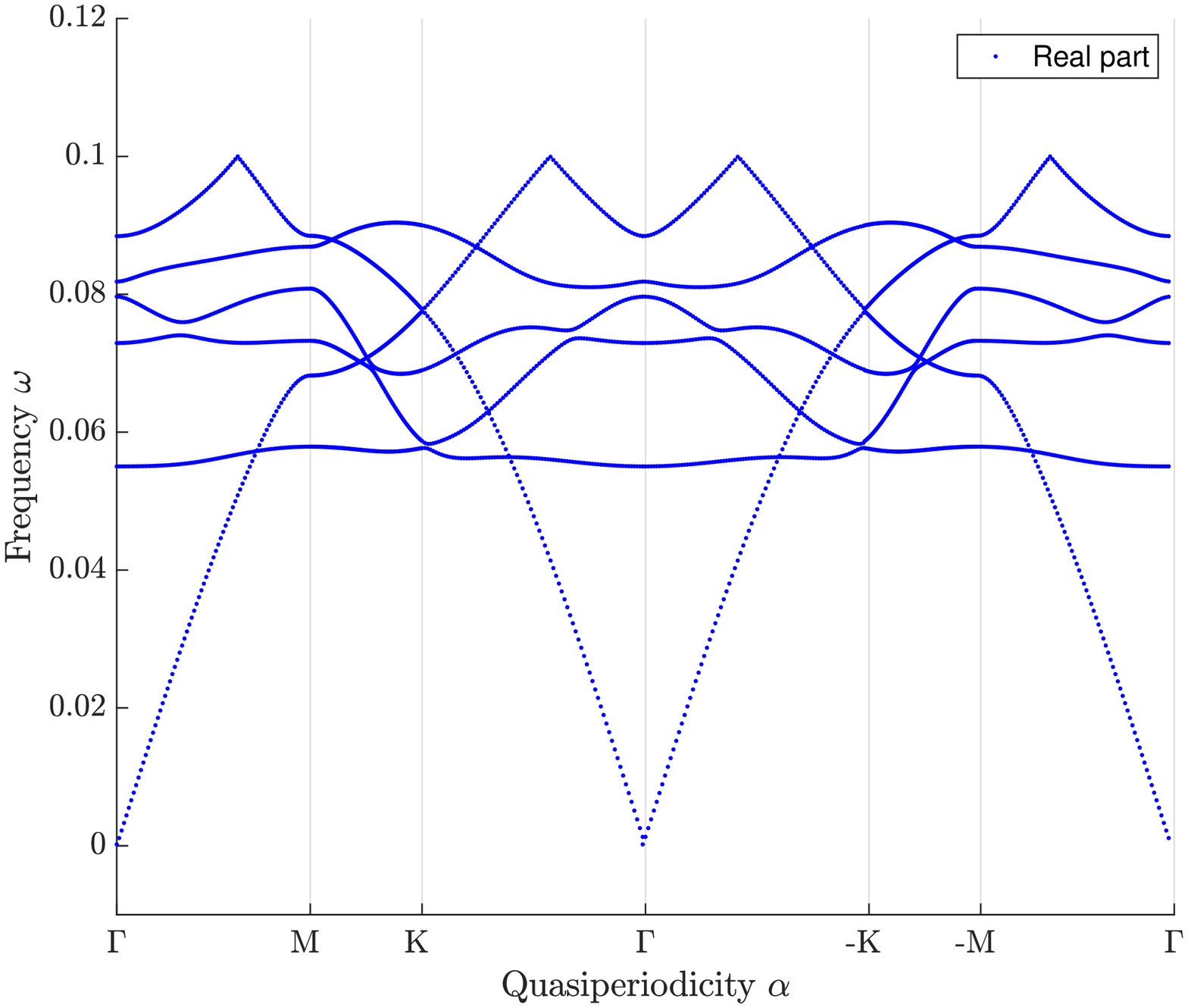}
		\end{center}
		\caption{ \centering
			Static case ($\varepsilon=0$).}
		\label{fig:HL_0}
	\end{subfigure}
	\hspace{10pt}
	\begin{subfigure}[b]{0.45\linewidth}
			\vspace{0pt}
		\begin{center}
\includegraphics[width=1\linewidth]{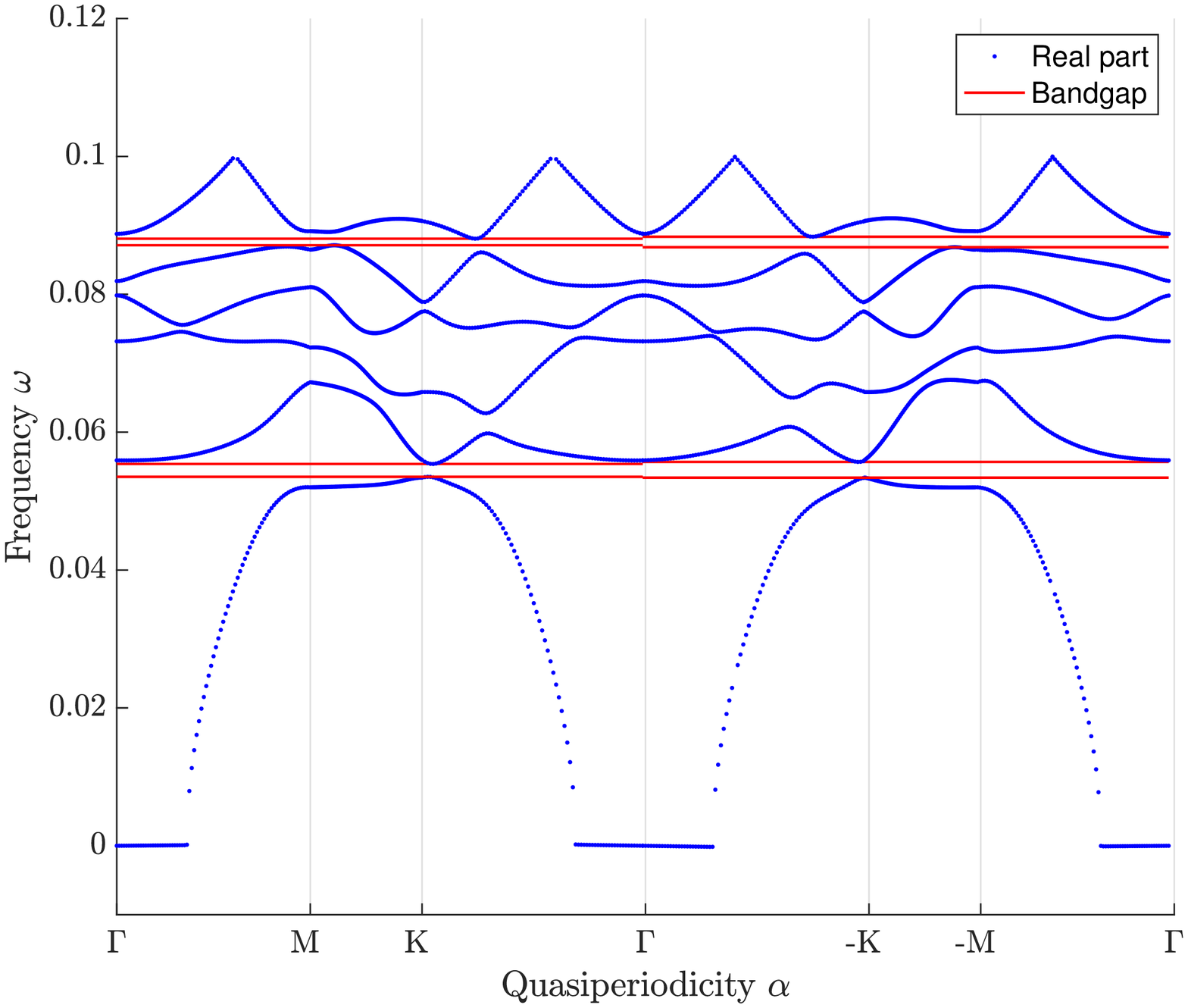}
		\end{center}
		\caption{\centering Modulated with $\varepsilon=0.25$.}
		\label{fig:HL_025}

	\end{subfigure}
	\centering
	\begin{subfigure}[b]{0.45\linewidth}
			\vspace{0pt}
		\begin{center}
 \includegraphics[width=1\linewidth]{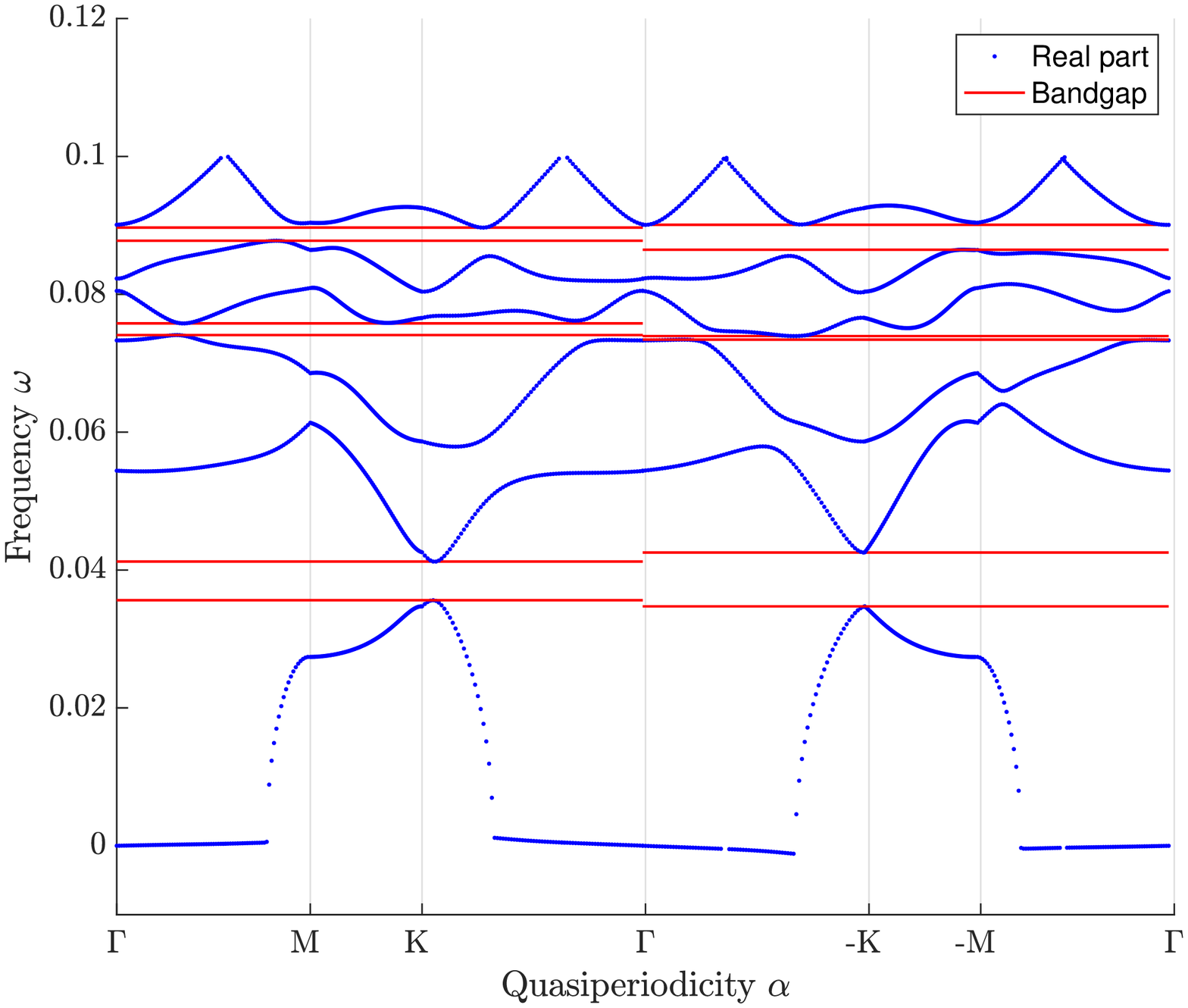}
		\end{center}
		\caption{\centering Modulated with $\varepsilon=0.5$.}
		\label{fig:HL_05}
	\end{subfigure}
	\hspace{10pt}
	\caption{Band structure of honeycomb lattice with six subwavelength resonators with modulation frequency $\Omega = 0.2$.} \label{fig7}
\end{figure}

In Figure \ref{fig7}, we compute the band structure with modulation frequency $\Omega = 0.2$. It is worth emphasizing that in this case, the numerical computation shows an even stronger occurrence of non-reciprocity compared with the chain and the square lattice. In fact, the second band gaps in the band structure of the honeycomb lattice are disjoint. Hence, there is a wave which can propagate in one path and not in the other and vice versa. This is not amounted to the first order effect discussed in the previous section. In the first order regime, the band gap opening resulted from $\varepsilon$-perturbation of $\alpha$ is either contained or contains the $-\alpha$ perturbation.	

	Furthermore, in the  unmodulated case (Figure \ref{fig:HL_0}), the band structure of the honeycomb   lattice shows a Dirac cone at the symmetry points $K$ and $-K$. By turning on the modulation (Figures \ref{fig:HL_025} and \ref{fig:HL_05}), the Dirac cones open up to local extrema of the band functions. The local extrema are called the valleys \cite{alexis}, or valley degrees of freedom.  By breaking reciprocity, we obtain different valleys for $K$ and $-K$. 

%
%
%
%
\section{Concluding remarks} \label{conclusion}

In summary, we have shown both analytically and numerically that time-modulated subwavelength resonators can lead to the emergence of unidirectional wave guiding properties associated with the presence of degenerate points in the band structure of the unmodulated periodic system by breaking time-reversal symmetry. We have also considered honeycomb lattices and illustrated a stronger occurrence of non-reciprocity compared with the cases of a chain and a square lattice.

Our results in this paper can be of immense importance for the mathematical foundation of other non-reciprocal guiding phenomena such as the valley Hall effect \cite{valley1, valley2, valley3, valley4, valley6, valley7} and the skin effect \cite{skin1, skin2, topo_skin}. The valley Hall effect may occur in truncated honeycomb lattices of time-modulated subwavelength resonators by opening non-reciprocal band gaps at Dirac points \cite{ammari2020highfrequency,ammari2020honeycomb}
while the skin effect may be obtained by opening non-reciprocal band gaps at exceptional points associated with the unmodulated structure \cite{ammari2020exceptional,ammari2020highorder,topo_exceptional}. These two challenging topics will be the subject of forthcoming publications. 

%

\appendix
\section{Eigenvalue perturbation theory and effective Hamiltonian}\label{app:pert}
Assume that $F=F_0+\varepsilon F_1+\varepsilon^2 F_2 + O(\varepsilon^3 )$ and $F_0$ is diagonal with respect to the basis vectors $w_1,\ldots,w_N$.  We would like to expand the eigenvalues of $F$ in terms of $\varepsilon$. This is a typical problem in perturbative quantum theory \cite{perturbation}. Similar formulas in quantum mechanical perturbation theory can be found in textbooks such as \cite{qmimperial}. The following derivation is reformulated to suit our setting.

We will focus on the perturbation of degenerate points. Let $f_0$ be a degenerate point of multiplicity $r$ and let  $w_1,\ldots, w_r$ be its associated eigenvectors. Without loss of generality, we assume that $(F_0)_{ii} = f_0$ for $i=1,\ldots,r$. We define the projection operator
\begin{equation*}
	P:=\begin{pmatrix}
		\text{Id}_r & \\ & 0
	\end{pmatrix}  \quad \text{and let } \ Q:=\text{Id}-P.
\end{equation*}
Here, $\text{Id}_r$ is the $r\times r$ identity matrix. 

We remark that $F_0$ commutes with $P$ and $Q$. Now, we fix an eigenvector $v_0\in\text{span}\{w_1,\ldots,w_n\}$ and expand $v$ and $f$ as follows
\begin{equation*}
	\begin{split}
		v&=v_0+\varepsilon v_1+\varepsilon^2 v_2+O(\varepsilon^3),\\
		f&=f_0+\varepsilon f_1+\varepsilon^2 f_2+O(\varepsilon^3).
	\end{split}
\end{equation*}
We require first that $v_0 = P(v)$, due to the normalization of $v$. From $Fv=fv$,  it follows that up to $O(\varepsilon^2)$ 
\begin{equation}
	\label{Qequation}
	\begin{split}
		& F_0v+\varepsilon(F_1+\varepsilon F_2)v=fv,\\
		& QF_0v+\varepsilon QVv=fQv,\\
		& Q(f \text{Id} -F_0)v=\varepsilon QVv \text{ and } \\
		& Qv =\varepsilon ((f \, \text{Id} -F_0)^{-1}Q)Vv,\end{split}	
\end{equation}
where $V:= F_1 + \varepsilon F_2$.

Note that we should treat $((f \, \text{Id} -F_0)^{-1}Q)$ as $0|_{E_{f_0}}\oplus ((f \, \text{Id} -F_0)^{-1}Q)|_{E_{f_0}^c}$, where ${E_{f_0}}$ denotes the eigenspace associated with $f_0$ and $E_{f_0}^c$ is its complementary. Similarly, we obtain that
\begin{equation*}
	PF_0v+\varepsilon PVv = fPv,
\end{equation*}
and therefore,
\begin{equation}
	\label{Pequation}
	f_0Pv+\varepsilon PVv = fPv, 
\end{equation}
where we have used that $PF_0v=F_0Pv=f_0Pv$. Now, we insert $v=Pv+Qv$ into the second term of the left-hand side of (\ref{Pequation}) and derive from $f_0Pv+\varepsilon PV(Pv+Qv)=fPv$ the following two identities:
\begin{equation}
	\label{master}
	\begin{split}
		& f_0Pv+\varepsilon PVPv+\varepsilon PVQv=fPv\text{ and }\\
		& f_0Pv+ \varepsilon PVPv+ \varepsilon^2PV\left((f \, \text{Id} -F_0)^{-1}Q\right)Vv=fPv.
	\end{split}
\end{equation}
For the $\varepsilon^2$-term, we evaluate the expression at $\varepsilon=0$:
\begin{equation}
	\label{hocus}
	PV((f\, \text{Id} -F_0)^{-1}Q)Vv|_{\varepsilon=0}=PF_1\left((f_0 \, \text{Id} -F_0)^{-1}Q\right)F_1v_0:=PF_1GF_1v_0,
\end{equation}
where $G:=(f_0 \, \text{Id} -F_0)^{-1}Q=\text{diag}(0,..,0,(f_0-\lambda_2)^{-1},\ldots,(f_0-\lambda_k)^{-1})$ if we assume that $F_0=\text{diag}(f_0,\ldots,f_0,\lambda_2,\ldots,\lambda_k)$. Hence, we can write that
\begin{equation}
	\label{effectiveH}	
	P\left(f_0 \, \text{Id} +\varepsilon (F_1+\varepsilon F_2)+\varepsilon^2(PF_1GF_1)\right)Pv_0=fv_0.
\end{equation}
With the so-called effective Hamiltonian:
\begin{equation}
	\label{effective}
	\mathcal{H}:=Pf_0P+\varepsilon PF_1P +\varepsilon^2P(F_1GF_1+F_2)P	,
\end{equation}
we can obtain $r$ perturbed eigenvalues up to order $\varepsilon^2$, if we know the form of $F_1$ and $F_2$.

\bibliographystyle{abbrv}
\bibliography{paper_reciprocal}
\end{document}